\documentclass[a4paper,12pt]{article}
\usepackage[left=4cm,right=1.4cm,top=4cm,bottom=2cm,bindingoffset=0cm]{geometry}
\usepackage[cp1251]{inputenc}
\usepackage[english]{babel}
\usepackage{amsmath}
\usepackage{amsthm}
\usepackage{amssymb}
\usepackage{tikz}
\usepackage{graphicx}
\usepackage{wrapfig}

\usetikzlibrary{patterns}

\newtheorem{theorem}{Theorem}
\newtheorem{lem}{Lemma}

\def\C{\mathbb{C}}
\def\i{i}

\def\x{X} 
\def\row{\operatorname{row}}
\def\col{\operatorname{col}}
\def\Row{\operatorname{Row}}
\def\Col{\operatorname{Col}}
\def\re{\operatorname{Re}}
\def\im{\operatorname{Im}}
\def\odd{\operatorname{odd}}
\def\pr{\operatorname{pr}}
\def\w{\mathfrak{Z}}
\def\ord{\operatorname{ord}}

\def\R{\mathbb{R}}
\def\Z{\mathbb{Z}}
\def\z{\mathcal{Z}}

\def\*{\otimes}

\begin{document}

\author{{\bf E.\,V.~Shchepin } \thanks{Steklov Mathematical institute scepin@mi-ras.ru}
}

\title{On Space-Filling Surfaces}

\date{October, 2026}

\maketitle
\begin{abstract}
We construct maps from cubes to higher-dimensional cubes with the largest possible H\"older exponent.
In particular, we construct a map from the square onto the cube with H\"older exponent $2/3$, thereby solving V. I. Arnold's problem.
\end{abstract}


\section{Introduction}
In 1890, Giuseppe Peano \cite{Peano} constructed continuous curves that fill a square and a cube.
This result caused a worldwide sensation at the time, prompting mathematicians to reconsider the meaning of a curve and
spurring the development of the topological theory of dimension.
Following Peano, Hilbert, Sierpinski, and many others constructed curves that fill a square.
Later, self-similar curves were constructed that fill cubes of arbitrary dimension.
These curves became known as Peano curves; the most famous example is Hilbert's curve.
Initially, Peano curves were regarded mainly as counterexamples in analysis. Much attention was paid to
their total lack of differentiability.

In the computer age, Peano curves gained renewed relevance and proved useful in discrete mathematics.
Their key property for computer applications is self-similarity, which makes it possible, on the one hand, to develop
efficient algorithms for approximating these curves by polygonal lines and, on the other hand, gives them remarkable geometric
properties. Peano-type curves have also been observed in nature: DNA is packed inside a cell in a way that can be described by a Peano-type curve.
Self-similar Peano curves are not merely continuous; they are also \emph{H\"older-continuous} and
attain the largest possible H\"older exponent among maps with the same image.

A map $f$ between metric spaces is called \emph{H\"older-continuous} with H\"older exponent $h$ and H\"older constant $\kappa$ if, for all $x,x\prime\in X$,
the following inequality holds:
\begin{equation}\label{hold-cont}
 |f(x)-f(x')|\le \kappa |x-x'|^h, \end{equation}
where the distance between points is denoted by the norm of their difference.

It is easy to prove that if a map from an $n$-dimensional cube into a Euclidean space of dimension $m>n$ has a H\"older exponent
greater than the ratio of the dimensions, $\frac{n}{m}$, then its image has zero $m$-dimensional volume and
is therefore nowhere dense and has topological dimension less than $m$.

All self-similar Peano curves attain the maximal possible exponent: planar curves have H\"older exponent
$\frac{1}{2}$, while curves filling three-dimensional cubes have exponent $\frac{1}{3}$.
Numerous computer applications of Peano curves are described in the book \cite{Bader}.

This article grew out of the author's attempt to answer the following question:

\emph{Does there exist a map from a square onto a cube with H\"older exponent $\frac23$?}

V. I. Arnold asked the author this question in 1988 (see \cite{Arnold2000}, problem 1988-5).
After considering it, the author proposed to Arnold the following construction of a self-similar map
from a square to a cube, analogous to the construction of Peano curves.

First, divide the square and the cube into the same number of cells (for example, 64 square and cubic cells), and establish
a one-to-one correspondence between the cells in the two partitions such that intersecting square cells correspond to intersecting cubic
cells. This property is called \,\emph{combinatorial continuity}.
Next, subdivide the square and cubic cells according to the same rule and establish
a one-to-one, combinatorially continuous correspondence between the second-order cells in corresponding first-order cells.
Continue in the same way at every subsequent stage.

If it could be completed, this construction would solve the problem.
However, even the first step proved to be quite difficult,
and overcoming these difficulties required the use of a computer.
In \cite{Tula64}, combinatorially continuous correspondences were presented between
cells in regular partitions of a square and a cube into 64 parts.

However, the second step proved impossible. Moreover, in 2010 the author proved
(see \cite{Sch2010}) that no self-similar map from a square to a cube exists.

A new impetus for solving Arnold's problem came from the author's reading of an article by R. Stong \cite{Stong1998},
which proposed a solution to the discrete analogue of Arnold's problem:

\emph{Does there exist a H\"older-continuous bijection from the two-dimensional
integer lattice $\Z^2$ onto the three-dimensional lattice $\Z^3$ with H\"older exponent $\frac23$?}

According to Stong, M. Friedman posed this question in the early 1990s.
The author was also motivated by potential computer applications of
fractal image compression, as described in \cite{Compression}.

Stong's construction is as follows. A certain cross-shaped, self-similar subset $X$ of the two-dimensional integer lattice is selected,
with fractal dimension $\frac{3}{2}$. An integer parametrization $\Z\to X$ is then constructed
that is a H\"older-continuous map with exponent $\frac{2}{3}$.
The desired H\"older-continuous map is obtained by composing the Cartesian square of the integer parametrization  with
a suitable linear projection from four-dimensional space to three-dimensional space.

The author's solution to Arnold's problem follows the same general approach. It is based on a flat, self-similar Zigzag curve $\z$, whose Hausdorff dimension is $\frac32$,
and which is homeomorphic to a segment. Its fractal (de Rham) parametrization $[0,1]\to\z$
is a H\"older-continuous map with exponent $\frac{2}{3}$.
The Cartesian square of the Zigzag lies in four-dimensional space and is projected linearly onto a three-dimensional space,
with a three-dimensional image. This solves Arnold's problem: the image of the Zigzag square under projection
contains a geometric cube, and the surrounding three-dimensional space admits a non-expansive retraction onto that cube.
Consequently, composing the product of two Zigzag parametrizations with the linear projection and the retraction yields
a H\"older-continuous map from a square onto a three-dimensional cube with exponent $\frac{2}{3}$.

Whereas the main difficulty in Stong's approach is parametrizing the cross-shaped lattice, in the author's approach
the main difficulty is proving that the projection of the Zigzag square has dimension three.
The author first tried to prove the existence of a suitable projection in general, but this approach did not
succeed. It was therefore necessary to prove the claim for a particular projection,
namely the same projection used by Stong. The result was not merely an existence theorem,
but an explicit construction suitable for applications.

In \cite{Stong1998}, Stong not only gave a positive solution to the Friedman problem but also proposed, for every
$n>1$, a construction of a bijection from an $n$-dimensional integer lattice to an $(n+1)$-dimensional lattice with H\"older exponent $\frac{n}{n+1}$.
In general, Stong's construction follows the same pattern: a cross-shaped subset of a lattice is constructed with fractal
dimension $\frac{n+1}{n}$, and its $n$th Cartesian power is mapped bijectively onto $\Z^{n+1}$
by the remarkable \emph{Stong map} $\sigma_n\colon \C^n\to\R^{n+1}$:
 \begin{equation}\label{Stong-map}
 \sigma_n(z_1,\dots,z_n)=(\re z_1,\dots,\re z_n, \sum_{k=1}^{n} 2^{k-1} \im z_k)
 \end{equation}

Compositions of the bijections he constructed yield H\"older-continuous bijections with exponent $\frac{n}{m}$ from an $n$-dimensional lattice to an $m$-dimensional lattice for any
$m>n>1$.
Unlike maps between cubes, whose H\"older exponent cannot exceed the ratio of the dimensions,
injective maps of lattices $\Z^n\to\Z^m$ have this ratio as the smallest possible exponent.

Arnold's problem is a special case of a more general question:
\emph{Does there exist a map from an $n$-dimensional cube
to an $m$-dimensional cube with H\"older exponent $\frac{n}{m}$?}

This article gives an explicit affirmative solution to the general problem.
As in Stong's construction, the general solution is a simple modification of the two-dimensional case.
Specifically, the author's solution is based on the so-called $r$-fold Zigzag, a
self-similar planar curve of Hausdorff dimension $\frac{r+2}{r+1}$ inscribed in a unit square.
Two copies of an $r$-fold Zigzag form a curve $\w_r$, called a \emph{Squiggle}.
Its natural fractal parametrization by the unit interval, $\w_r(t)$, is H\"older-continuous with exponent $\frac{r+1}{r+2}$.

We regard the $r$-Squiggle as lying in the complex plane $\C$, with projection onto the real axis equal to
$[0,1]$ and projection onto the imaginary axis equal to $[0,2]$.
The Cartesian $(r+1)$-fold power $\w_r^{r+1}$ is bi-H\"older-homeomorphic to a fractal cube of dimension $r+1$
and Hausdorff dimension $r+2$, lying in $\C^{r+1}$. Applying the Stong map $\sigma_{r+1}$ and then projecting
$\R^{r+2}\to\R^{r+1}\times\R/2\Z$ yields a covering of the cylinder $[0,1]^{r+1}\times\R/2\Z$.
Thus, an explicit H\"older-continuous map $\alpha_n\colon [0,1]^n\to [0,1]^{n+1}$ with exponent $\frac{n}{n+1}$ is given by
 \begin{equation}\label{main-form}
 \alpha_n(t_1,\dots, t_n)=(\re\w_{n-1}(t_1),\dots, \re\w_{n-1}(t_n), \left|\sum_{k=1}^{n} 2^{k-1}\im \w_{n-1}(t_k)\bmod 2 -1\right|)
 \end{equation}
In this formula, the image of a point under $\alpha_n$ can be computed efficiently, as can a preimage of a point in the image.

The resulting theorem was announced by the author in \cite{Sch26}.
 \begin{theorem}\label{th1}
For any integers $m>n>1$, there exists a H\"older-continuous map from an $n$-dimensional cube onto an $m$-dimensional cube with exponent $\frac{n}{m}$.
 \end{theorem}

The article also presents two alternative constructions of sharp H\"older maps $I^{r+1}\to I^{r+1}\times S^1$,
based on a fractal Cross, a compact analogue of Stong's discrete cross.
The fractal Cross and the corresponding solution to Arnold's problem were also obtained independently by Badger and Palmer \cite{Palmer}.

Like a Zigzag, an $r$-fold Cross is a self-similar compact set with Hausdorff dimension $\frac{r+2}{r+1}$.
However, unlike the Zigzag and the Squiggle, the Cross is not a Jordan curve. The fractal parametrizations $\tau_r$ of the Cross constructed by the author
have points of multiplicity four. The extremal H\"older map $I^{r+1}\to I^{r+1}\times S^1$,
constructed using an $r$-fold Cross inscribed in the unit square, is given by
 \begin{equation}\label{formula-x}
 X (t_1,\dots,t_{r+1})= (\re \tau_r(t_1),\dots,\re \tau_r(t_{r+1}),\sum_{k=1}^{r+1} 2^{k-1}\im \tau_r(t_k) \mod 1)
\end{equation}

Finally, the third construction combines a Cross and a Squiggle of odd multiplicities $2r-1$ and gives maps from squares and even-dimensional cubes onto the corresponding cylinders:

 \begin{multline*}
 Y (t_1,\dots,t_{2r})= (\re \w_{2r-1}(t_1),\re \tau_{2r-1}(t_2), \dots,
 \re \w_{2r-1}(t_{2r-1}),\re \tau_{2r-1}(t_{2r}),\\
 \sum_{k=1}^{r} 4^{k-1}\im (\w_{2r-1}(t_{2k-1})+\tau_{2r-1}(t_{2k}) ) \mod 1)
\end{multline*}

\section{Fractal Zigzag}

\subsection{Square partitions.}

\paragraph{ Segment partitions.}

A regular partition of a segment $I\subset\R$ is a covering of $I$ by equal-length subsegments with pairwise disjoint interiors.

A regular partition of $I$ into $n$ subsegments is called an $n$-partition,
denoted by $I/n$; the number $n$ is called the \emph{order} of the partition.
The subsegments of the partition $I/n$ are called its \emph{cells}.
If a segment $J$ is a cell in a regular partition of a containing segment $I$, we denote the partition by $I/J$ and say that $J$ \emph{divides} $I$.
The cells at the two ends of a partition are called \emph{extremal}.
We denote the length of a segment $I$ by $|I|$; thus, the order of the partition $I/J$ is $|I|/|J|$.

A segment is called \emph{oriented} if its endpoints are ordered.
The first endpoint is called the beginning of the segment, and the second is called its end.
The number of cells between the beginning of an oriented segment $I$ and a given cell $J$ is called the \emph{position} of $J$ on $I$.
Thus, the initial cell (the one containing the beginning of the segment) has position zero, and the final cell has position one less than the order of the partition.

\paragraph{Partitioning rectangles.}

We equip the Cartesian plane $\R^2$ with a fixed Cartesian coordinate system.
We assume that the $x$-axis is horizontal and the $y$-axis is vertical.
The distance between points $x,y\in\R^2$ is denoted by $|x-y|$.
We use two metrics on the plane: the Euclidean metric and the $l_\infty$ metric.
Since we mainly work with segments parallel to the coordinate axes, whose lengths
are the same in both metrics, the choice of metric is usually immaterial; when it matters, we specify it.

A rectangle with sides parallel to the coordinate axes is called \emph{Cartesian}.
Every Cartesian rectangle is the Cartesian product of two segments: its horizontal
(abscissa-axis) and vertical (ordinate-axis) projections.

For a Cartesian rectangle $P=I\times J$, the \emph{regular partition} of order $m\times n$ (or an $(m\times n)$-partition)
is the set $P/(m\times n)=\{F\times G\mid F\in I/m,G\in J/n\}$.
For a cell $F\in I/m$, the set $\{F\times G\mid G\in J/n\}$ is called a \emph{column}
of the partition $(I\times J)/(m\times n)$ with base $F$.
Similarly, for a cell $G\in J/n$, the set $\{F\times G\mid F\in I/m\}$ is called a \emph{row} of the partition with base $G$.
If $P$ is a cell of a regular partition of a containing Cartesian rectangle $P\prime$, we say that $P$ divides $P\prime$ and denote the partition by $P\prime/P$.
The order of this partition is uniquely determined by the ratios of the corresponding side lengths.

We denote the number of columns of $P\prime/P$ to the left of the column containing $P$ by
$\col(P,P\prime)$, or simply $\col P$ when the containing rectangle is clear from context.
This number is called the \emph{horizontal position} of $P$ in $P\prime$.
Similarly, $\row(P,P\prime)$ denotes the number of rows of $P\prime/P$ below the row containing $P$,
and is called the \emph{vertical position} of $P$ in $P\prime$.

In this article, we mainly consider \emph{Cartesian squares}, i.e., Cartesian rectangles that are squares.
A regular partition of a Cartesian square is a partition of type $n\times n$, in which all cells
are also squares.
A cell in a regular partition of a Cartesian square $Q$ is said to lie on a diagonal of $Q$
if its intersection with the diagonal is exactly a diagonal of that cell.
The diagonal from the lower-left to the upper-right corner is called the \emph{ascending diagonal}; the other diagonal is called the \emph{descending diagonal}.

\paragraph{Parity.}
A regular partition of a Cartesian square is called even or odd according to the parity of the number of cells along each side.
Thus, $Q/(n\times n)$ is even if $n$ is even and odd otherwise. In what follows,
we consider only even partitions.

We define the \emph{parity} of a cell $F$ in a regular partition of the Cartesian square $Q$ by
 \begin{equation}\label{odd-frac}
 \odd (F,Q)=(\col (F,Q)+\row (F,Q))\bmod 2
 \end{equation}
If $\odd(F,Q)=0$, the cell is called even with respect to $Q$; otherwise, it is called odd.

\begin{lem}\label{odd-even}
Every cell lying on the ascending diagonal of a Cartesian square is even in any regular partition of that square.
Every cell of an even partition that lies on the descending diagonal is odd.
\end{lem}

 For a Cartesian square $Q$, we denote the length of its side by $|Q|$.
 \begin{lem}\label{3q}
Suppose that three Cartesian squares $Q_1\subset Q_2\subset Q_3$ satisfy $Q_1$ divides $Q_2$ and $Q_2$ divides $Q_3$. Then
$Q_1$ divides $Q_3$, and the following equalities hold:
 \begin{gather*}
  \row (Q_1,Q_3)= \row (Q_1,Q_2)+(|Q_2|/|Q_1|)\row (Q_2,Q_3) \\ \col (Q_1,Q_3)= \col (Q_1,Q_2)+(|Q_2|/|Q_1|)\col (Q_2,Q_3)
 \end{gather*}
 \end{lem}
 \begin{proof}
The proof is left to the reader.
 \end{proof}

It follows from this lemma that the parities of $Q_1$ relative to the containing squares $Q_2$ and $Q_3$ agree whenever $|Q_2|/|Q_1|$ is even.

\paragraph{Square chains.}

For any planar set $M$, let $[M]$ denote the smallest Cartesian rectangle containing it; we call this the
\emph{circumscribed rectangle} of $M$.
Thus, for two points $a,b$ in the plane, the circumscribed rectangle $[a,b]$ is the segment joining them if the points
lie on the same vertical or horizontal line; otherwise, it is a rectangle.

A square with one selected vertex is called a \emph{pointed square}, and the selected vertex is its
\emph{initial vertex}. The diagonal through the initial vertex is the \emph{main diagonal}; its endpoints are the main vertices.
The main vertex other than the initial vertex is called the terminal vertex of the pointed square.
The notation $[V,W]$ for a pointed square means that $V$ is its initial vertex and $W$ its terminal vertex.

A sequence $F_*=(F_1,\dots,F_m)$ of pointed Cartesian squares of equal size is a \emph{square chain} of length $m$ if, for every $i<m$, $F_i\cap F_{i+1}$ is a common vertex or side and the initial vertex of $F_{i+1}$ is the terminal vertex of $F_i$. The squares in the chain are called its cells.
The \emph{set} of a square chain is the union of its cells.
The Cartesian rectangle circumscribed about the set of $F_*$ is denoted by $[F_*]$.
The beginning of the chain is the beginning of its first cell, and the end is the end of its last cell.

The sequence of main diagonals of a square chain forms a broken line, called the chain's \emph{diagonal line}.
The initial and final vertices of the chain are called its \emph{extreme points}.

\paragraph{Central code.}
The segment joining the centers of two adjacent cells in a square chain lies either along the extension of the diagonal of the preceding cell or parallel to one of its sides.
We encode horizontal, vertical, and diagonal directions by the symbols $-$, $|$, and $\times$, respectively.
Thus, every square chain $F_*$ of length $m$ determines a word of length $m-1$ over the three-symbol
\emph{direction alphabet} $\{|,-,\times\}$, called the chain's \emph{central code} and denoted by $C(F_*)$.

The $i$th symbol of $C(F_*)$ specifies the direction of the segment joining the centers of $F_i$ and $F_{i+1}$.
If the chain $F_*$ has only one cell, its central code is considered \emph{empty}.
For a code $A$, we write $|A|$ for its \emph{chain length}, i.e., the length of the corresponding chain,
which is one greater than the length of the code itself.

The central code, initial point, and first cell uniquely determine the entire chain, as stated in the following lemma.
determine the entire chain by virtue of the following lemma.

\begin{lem}\label{code2chain}
Let $F_1$ be a pointed square and $C_*$ any central code. Then there is a unique square chain $F_*$ starting with $F_1$, with the same initial point as $F_1$ and central code $C_*$.
\end{lem}
\begin{proof}
We determine the cells of the chain successively. The first cell is fixed. Suppose that
a chain segment of length $n$ has already been uniquely determined. To choose $F_{n+1}$,
only three squares of the same size as $F_n$ can have the terminal vertex of the current chain as a vertex.
The symbol $C_n$ uniquely identifies which of these three squares to choose.
\end{proof}

\paragraph{Code addition.}
For a code $A$ and an integer $i<|A|$, let $A[i]$ denote its $i$th symbol. We also adopt the following \emph{empty-code convention}:
For any non-integer real number $x$ in the interval $[1,|A|-1]$, set
\begin{equation}\label{emptyset}
A[x]=\emptyset
 \end{equation}
For any real numbers $x,y$, let $A[x,y]$ denote the subword of $A$ consisting of symbols whose indices lie in $[x,y]$.

The \emph{sum} of codes $A+B$ is defined by concatenation. With the empty-code convention, it is given by
\begin{equation}\label{conact}
 (A+B)[x]=A[x]+B[x-|A|+1]
\end{equation}
The chain length of the sum is one less than the sum of the chain lengths: $|A+B|=|A|+|B|-1$.

\paragraph{Parity.}
A code $A$ is called \emph{even} or \emph{odd} according to whether the number of its non-diagonal symbols is even or odd.
For $i\ge1$, let $\odd(A[i])$ denote the number of non-diagonal symbols in the subword $A[1,i]$, reduced modulo $2$,
and set $\odd A=\odd A[|A|]$. Thus, the parity of a code is determined by $\odd A$. We regard $\odd$ as Boolean-valued, so its values are added modulo $2$.

\begin{lem}\label{code-odd}
For a diagonally continuous chain $(F_1,\dots,F_n)$ with central code $A$,
$\odd(F_n)=\odd(F_1)+\odd A$.
\end{lem}
\begin{proof}
Indeed, if the segment joining the centers of adjacent cells is vertical or horizontal, then the cells have
opposite parity; if it is diagonal, they have the same parity. Hence, the number of parity changes along the chain equals the number of non-diagonal symbols in its code.
\end{proof}

\paragraph{Orthogonalization.}
The symbols for the horizontal and vertical directions are called \emph{orthogonal} to each other.
The diagonal direction is considered orthogonal to itself. For a direction $a$, its orthogonal direction is denoted by
$\perp(a)$. Thus, $\perp(|)=-$, $\perp(-)=|$, and $\perp(\times)=\times$.

For a code $A$, define the orthogonal code $\perp(A)$ symbolwise by $\perp(A[i])$.
We write $\perp^n$ for the $n$-fold iteration of $\perp$. Thus, $\perp^n$ is the identity for even $n$ and equals $\perp$ for odd $n$.
with $\perp$ for odd.

The parity of a code agrees with that of its orthogonal code; in particular, for every $i$,
\begin{equation}\label{odd-perp}
 \odd(\perp A [1,i])=\odd A[1,i]
\end{equation}

\paragraph{ Superposition of codes.}

The \emph{superposition} of central codes is defined by the following formula:
 \begin{equation}\label{supercode}
 A\circ B= \sum_{k=0}^{|A|-1}(A[k]+ \perp^{\odd A[k]} B)
 \end{equation}
If $A$ is empty, then $A\circ B=B=B\circ A$ for every code $B$.
The chain length of the superposition is the product of the chain lengths: $|A\circ B|=|A||B|$.
The $i$th symbol of the superposition is given by
\begin{equation}\label{super-i}
 (A\circ B)[i]=A[i/|B|]+\perp^{\odd A[i/|B|]}B [i \bmod |B|]
\end{equation}

\begin{lem}\label{super-perp}
For central codes $A$ and $B$, the following equality holds:
$$\perp(A\circ B) = \perp( A) \circ \perp (B)$$
\end{lem}
\begin{proof}
If $i$ is divisible by $|B|$, then the $i$th symbols on both sides equal $\perp(A[i/|B|])$.

Otherwise, the $i$th symbol is
 $$\perp^{\odd(\perp A[i/|B|])} \perp (B[i \bmod |B|])$$
for the right one.
These expressions agree because $\odd(\perp A[i])=\odd A[i]$.
\end{proof}

\begin{lem}\label{code-ass}
If $\odd B=0$, then for any codes $A$ and $C$,
 $$A\circ(B\circ C)=(A\circ B)\circ C$$
\end{lem}
\begin{proof}
If $i$ is divisible by $|B||C|$, then both sides have $i$th symbol $A[i]$.

If $i$ is divisible by $|C|$ but not by $|B||C|$,
then the $i$th symbols of $A\circ(B\circ C)$ and $(A\circ B)\circ C$ are either $B[i\bmod (|B||C|)/|C|]$ or $\perp B[i\bmod (|B||C|)/|C|]$,
depending on the parity of the corresponding prefix of $A$.

If $i$ is not divisible by $|C|$, both expressions take one of the values
$C[i\bmod |C|]$ and $\perp C[i\bmod |C|]$, depending respectively on the parity of a prefix of $A\circ B$ and
on the sum of the parities of the corresponding prefixes of $A$ and $B$.

The assumption that $B$ is even gives
$$\odd(A\circ B)[1,i/|C|]=\odd A[1, i/(|B||C|)]+ \odd(B[1,(i\bmod C)/|C|]),$$
because when computing the parity of the prefix $(A\circ B)[1,i/|C|]$, all complete copies of $B$ contribute zero parity;
only the relevant prefix of $A$ and an incomplete copy of $B$ remain.
\end{proof}

\paragraph{Superposition of chains.}
A square chain $F_*$ is called \emph{inscribed} in the pointed square $Q$ if:
\begin{itemize}
\item the initial vertex of the chain is the initial vertex of $Q$;
\item the final vertex of the chain is the final vertex of $Q$;
\item the circumscribed rectangle of the chain is $Q$.
\end{itemize}
The type of diagonal joining the initial and final vertices of a square chain is called its \emph{diagonal type}.
Square chains are called ascending or descending according to their diagonal type.

The central code $C$ of a square chain inscribed in a square is called a \emph{square code}.
Chains generated by the same central code and starting in different pointed Cartesian squares are related by a combination of homothety,
translation, and rotation.
These transformations map each chain to the corresponding chain, preserving its combinatorial structure.
Therefore, if the circumscribed rectangle of one such chain is a square, then the circumscribed rectangles of the others are also squares.

\begin{lem}\label{code-diag}
Let $A$ be the central code of the square chain $(F_0,\dots,F_m)$,
and let $B$ be a square code.
Then there is a uniquely determined chain $G_*$ of length $|B|(m+1)$ such that:
\begin{itemize}
\item the initial point of $G_*$ is the initial point of $F_*$;
\item the set of $G_*$ is contained in the set of $F_*$;
\item the central code of $G_*$ is the superposition $A\circ B$.
 \end{itemize}
\end{lem}
\begin{proof}
The proof proceeds by induction on the length of $F_*$, using Lemma \ref{code2chain}.
Suppose that the construction of $G_*$ has been completed over the segment $F_0,\dots,F_k$. Consider the square $F_{k+1}$.
If $A[k+1]$ is diagonal,
translating the part of $G_*$ in $F_k$ to $F_{k+1}$ gives the next segment of the chain.
If $A[k+1]$ is vertical or horizontal, then, in addition to translating $F_k$ to $F_{k+1}$,
we must rotate the corresponding chain segment by $90^\circ$ about the center of $F_{k+1}$ so as to
align its initial point with the common extreme vertex of $F_k$ and $F_{k+1}$.
 \end{proof}

We record the following consequence of the lemma.
\begin{lem}\label{sq2chain}
For any Cartesian square $Q$ and any square code $C$, there is a unique square chain inscribed in $Q$ with central code $C$ and starting at a specified vertex of $Q$.

 \end{lem}
 \begin{proof}
Indeed, it suffices to apply Lemma \ref{code-diag} to the case in which $F_*$ consists of a single cell.

 \end{proof}

The chain code $C$ is called \emph{symmetric} if $C[i]=C[|C|-i]$ for every $i<|C|$.

For a pointed Cartesian square $Q$ and a square central code $C$, let $Q(C)$ denote
the set of the square chain inscribed in $Q$.
If $C$ is symmetric, the sets of chains with code $C$ starting at opposite vertices of $Q$
coincide. Thus, the notation $Q(C)$ is unambiguous for any Cartesian square with the specified main diagonal,
without specifying which main vertex is the initial one.

 \subsection{ Zigzag}

\paragraph{Zigzag Prototype.}
The codes $Z^+=(--|\times|--)$ and $Z^-=(||-\times-||)$ are called the horizontal and vertical
Zigzag codes, respectively. They are symmetric, mutually orthogonal, and square codes.
They are realized by square chains in the $4\times4$ partition of a Cartesian square $Q$.
The ascending chain with the horizontal code and the descending chain with the vertical code
can be visualized using the following matrices: $$\begin{matrix}
  x & 5 & 6 & 7 \\
  x & 4 & x & x \\
  x & x & 3 & x \\
  0 & 1 & 2 & x
 \end{matrix} \qquad
\begin{matrix}
 0 & x & x & x \\
 1 & x & 4 & 5 \\
 2 & 3 & x & 6 \\
 x & x & x & 7
\end{matrix} $$

The cells of each chain are numbered according to their positions in the square partition. The set of the chain on the left is called the ascending Zigzag prototype, and the set on the right is the descending Zigzag prototype.
In what follows, we do not consider ascending chains with a vertical Zigzag code or descending chains with a horizontal code.
We use $Q(Z^+)$ to denote the ascending chain with central code $Z^+$ inscribed in $Q$.
There are two such chains: one starts at the lower-left corner and the other at the upper-right corner. The sets of these
chains coincide, but their orders are reversed.
Similarly, $Q(Z^-)$ denotes either descending chain with vertical Zigzag code.

Reversing the order of a chain preserves the relation ``lies between: if a cell lies between two others in one order,
it also lies between them in the reverse order.
The number of cells between two cells in the chain, counted along the chain, is called their \emph{chain distance}.
This notion allows us to introduce several concepts without specifying the initial point or diagonal type of $Q(Z^\pm)$.
Two cells are \emph{neighbors} if their chain distance is one. The cells are called \emph{extreme} if
all other cells lie between them, and \emph{central} if their chain distances from the extreme cells differ
by at most one.
 \begin{lem}\label{proto1}
 If the non-adjacent cells of the Zigzag prototype intersect, then they have the same parity.
 \end{lem}
 \begin{proof}
Indeed, the matrices representing the Zigzag codes show that a pair of intersecting, non-adjacent cells
can consist only of a central cell and a cell adjacent to an extreme cell.
 \end{proof}
 \begin{lem}\label{proto2}
Let $F$ and $G$ be cells of the partition $Q/(2\times2)$ of a Cartesian square $Q$,
lying on the same diagonal, and let $Z$ be a Zigzag code such that the common vertex of $F$ and $G$ does not belong to either chain set
$F(Z)$ or $G(Z)$.
Then the distance between these two sets is at least one quarter of the side length of $Q$.
 \end{lem}
 \begin{proof}
Assume that $F$ and $G$ lie on the ascending diagonal of $Q$, with $F$ below and to the left of $G$.
Let $F\prime$ and $G\prime$ be the cells in the $4\times4$ partitions of $F$ and $G$, respectively, that intersect each other.
Then neither $F\prime$ nor $G\prime$ belongs to its corresponding Zigzag chain, and both are descending cells.
Thus, they have vertical code. The cell $G\prime\prime$ in $G/(4\times4)$ immediately to the right of $G\prime$
does not belong to $G(Z)$, and the cell $F\prime\prime$ immediately to the left of $F\prime$
does not belong to $F(Z)$. This proves the required inequality.
 \end{proof}
\paragraph{Building a Zigzag.}

We denote the $n$-fold superpositions of the horizontal and vertical Zigzag codes by $Z^{+n}$ and $Z^{-n}$, respectively.
These codes are also symmetric and mutually orthogonal.

For any Cartesian square $Q$, define the $n$-fold prototype of the ascending (respectively, descending) Zigzag
as $Q(Z^{+n})$ (respectively, $Q(Z^{-n})$),
the set of the square chain inscribed in $Q$ with central code $Z^{+n}$ (respectively, $Z^{-n}$) and endpoints on the ascending (respectively, descending) diagonal
of $Q$.
Rotating $Q$ by $90^\circ$ maps the ascending Zigzag prototype to the descending one,
and vice versa.

The ascending (respectively, descending) Zigzag inscribed in $Q$ is defined as the intersection
of all $n$-fold ascending (respectively, descending) Zigzag prototypes inscribed in $Q$.

 \begin{equation}\label{zigzag-def}
 Q(Z^{\pm\infty})=\bigcap_{n=1}^\infty Q(Z^{\pm n})
 \end{equation}
We use the term Zigzag without a qualifier when either of the two types is intended.
it refers to either of the two Zigzags.

\centerline{\includegraphics[scale=2]{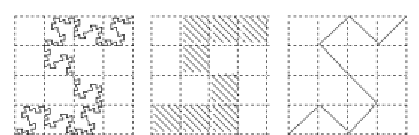}}
 \centerline{Figure 1 }

The figure shows an ascending Zigzag, its prototype, and the diagonal polyline of the prototype.
Note that the Zigzag pattern in the article \cite{Sch26} is incorrect.

The intersection of a Zigzag with a square cell of its $n$-fold prototype is called a \emph{Zigzag cell of order $n$}.
The multiple Zigzag prototypes contain the extreme points of the first-order prototypes; hence,
both Zigzags contain the extreme points of their prototypes, which we call the \emph{extreme points of the Zigzags}.
It follows that, for every order-$n$ cell $F$ of a Zigzag, its circumscribed square $[F]$ is a cell of the $n$-fold prototype.
The cell containing an extreme point of the Zigzag is called an \emph{extreme cell}.

\begin{lem} If the coordinates of all vertices of the cells in the $n$-fold Zigzag prototype are integers, then
the sums of the coordinates of all extreme points of order-$n$ Zigzag cells have the same parity.
 \label{ext-points}
\end{lem}
\begin{proof}
Indeed, the extreme points of a cell have the same parity. Neighboring cells share an endpoint,
so all their extreme points have the same parity. Since any two cells are connected by a chain of neighboring cells, the claim follows.
\end{proof}

\begin{lem}\label{zig-frac-odd}
Every even cell of a Zigzag is homothetic to the ascending Zigzag,
and every odd cell is homothetic to the descending Zigzag.
\end{lem}
\begin{proof}
Homothety maps the prototype of a multiple Zigzag in any cell $F$ to the corresponding prototype in its circumscribed square $Q$.
The parity of a cell in the prototype in $[F]$, measured relative to $Q$, agrees with that of its homothetic image. The same therefore holds for intersections of multiple prototypes.
\end{proof}

\begin{lem} The ascending Zigzag meets the vertical sides of its circumscribed square only at its extreme points.
The descending Zigzag meets the horizontal sides of its circumscribed square only at its extreme points.
 \label{zig-side}
\end{lem}
\begin{proof}
For an ascending Zigzag prototype of any multiplicity, only the extreme cells meet the vertical sides of the circumscribed square; hence,
the intersections of the multiple prototypes with a vertical side form a nested sequence of segments
whose lengths tend to zero. Their intersection is therefore the extreme point of the prototype.

The claim for a descending Zigzag follows by symmetry, since it is obtained from an ascending Zigzag by a rotation through $90^\circ$.
\end{proof}

\begin{lem}\label{frac-intersect}
The intersection of any two distinct Zigzag cells of the same order is either empty or consists of their common extreme point.
\end{lem}
\begin{proof}
It suffices to consider cells of the ascending Zigzag. Let two such cells intersect. Their circumscribed squares either share a side or a vertex.
In the first case, the cells have opposite parity. If the common side is vertical, the even cell meets it at an extreme point; if it is horizontal, the odd cell does.
If the circumscribed squares share a vertex, the cells have the same parity, and the shared vertex is
an extreme point of both cells.
\end{proof}

Distinct cells of the same order in a Zigzag with nonempty intersection are called \emph{neighbors}.

\begin{lem} Cells of an order-$n$ Zigzag are neighbors if and only if their indices in the chain $Q(Z^\pm)$ differ by one.
\label{zig-sosed}
\end{lem}
\begin{proof}
We prove the claim by induction on $n$.
If the indices of two cells in the Zigzag prototype chain differ by one, then the cells share an endpoint.
If the indices of cells $F$ and $G$ differ by more than one, their circumscribed squares intersect in only two cases:
when the indices are $1,3$ or $4,6$. In either case, the squares meet only at a vertex, which does not belong to the Zigzag sets
inscribed in those squares.

Assume the claim holds for order-$n$ cells. Let $F$ be a cell of order $n+1$ contained in a first-order cell $G$.
Apply a homothety mapping $G$ onto a Zigzag to map $F$ to an order-$n$ cell.
The induction hypothesis applies to that image.
If $F$ is not an extreme cell of $G$, its image has exactly two neighboring cells,
whose prototypes give the desired neighbors of $F$; it has no neighbors outside $G$.
If $F$ is an extreme cell of $G$, its image is also extreme in the Zigzag and hence $F$
has exactly one neighbor inside $G$. If the common extreme point $E$ of $F$ and $G$ is also an extreme point of the Zigzag,
then $F$ has no neighbor outside $G$. Otherwise,
the cell $G\prime$ adjacent to $G$ intersects $G$ exactly at $E$, so the order-$(n+1)$ cell in $G\prime$ containing $E$
is the second neighbor of $F$ and has no other neighbors.
\end{proof}

\subsection{Multiple Zigzag.}
In this section, $r$ denotes the natural number specifying the multiplicity of the Zigzag.
We denote by $\times_r$ the chain code of length $2^{r-1}$ consisting entirely of diagonal symbols.
In particular, the code $\times_1$ is empty, and $\times_2=(\times)$.
The chain codes of the $r$-fold Zigzag prototypes are given by the following formulas:
\begin{equation}\label{r-proto}
 Z^+_r=\times_r\circ (--|\times|--)\qquad Z^-_r=\times_r\circ (||-\times-||)
\end{equation}
The codes $Z_r^+$ and $Z_r^-$ are symmetric and mutually orthogonal.

The square chains $Q(Z_r^+)$ and $Q(Z_r^-)$ consist of cells
in the regular partition $Q/(2^{r+1}\times2^{r+1})$ of the Cartesian square $Q$.
The sets of these chains are called the \emph{prototypes of the $r$-fold ascending and descending Zigzags}, respectively.

The prototype of an $r$-fold Zigzag (an \emph{$r$-Zigzag} for short) is divided into eight groups of cells,
called \emph{blocks}; each block is a Zigzag prototype inscribed
in a diagonal cell of the partition $Q/(2^{r-1}\times2^{r-1})$.
This diagonal cell is called the \emph{block box}.
For the ascending prototype, the block boxes are the ascending diagonal cells; for the descending prototype, they are the descending ones.
Thus, the total number of cells in the $r$-fold prototype is $8\cdot2^{r-1}=2^{r+2}$.

When $r=1$, this definition of a prototype agrees with the earlier definition of a Zigzag prototype.
of Zigzag prototypes.
\begin{lem}\label{povorot}
Rotating $Q$ by $90^\circ$ about its center maps $Q(Z_r^+)$ to $Q(Z_r^-)$ and vice versa; every block
of the ascending prototype is mapped to a block of the descending prototype, and conversely.
\end{lem}
\begin{proof}
A rotation by $90^\circ$ maps the ascending diagonal of $Q$ to the descending diagonal. It therefore maps ascending block boxes
to descending block boxes and vice versa. The lemma follows because the $90^\circ$ rotation of an ascending Zigzag prototype
is a descending prototype, and conversely.
\end{proof}

We denote the $n$-fold superpositions of $Z_r^+$ and $Z_r^-$ by $Z_r^{+n}$ and $Z_r^{-n}$, respectively.
The corresponding chains $Q(Z_r^{+n})$ and $Q(Z_r^{-n})$ inscribed in $Q$ are called the
$n$-fold prototypes of the ascending and descending Zigzags.

Every cell $F$ in an $(n+1)$-fold prototype of either the ascending or descending $r$-Zigzag is contained in a single cell
of the corresponding $n$-fold prototype. We denote that cell by $\pr(F)$ and call it the \emph{parent} of $F$.
Cells of an $(n+1)$-fold prototype belong to the same block if they have the same parent cell in the $n$-fold prototype
and lie in the same block of that parent cell.
Each block crosses the boundary of its block box at two vertices lying on its diagonal. Depending
on the diagonal, a block is called \emph{ascending} or \emph{descending}.

Finally, the ascending (respectively, descending) $r$-fold Zigzag $Q(Z_r^{+\infty})$ (respectively, $Q(Z_r^{-\infty})$) inscribed in $Q$ is defined as the intersection of all its
ascending (respectively, descending) prototypes.
$$ Q(Z_r^{\pm\infty})=\bigcap_{n=1}^\infty Q(Z_r^{\pm n}) $$
Note that
\begin{equation}\label{opis}
 [Q(Z_r^{+\infty})]= [Q(Z_r^{-\infty})]= Q
\end{equation}
Thus, the two possible meanings of the phrase ``the square $Q$ is circumscribed about the $r$-Zigzag are equivalent.
Indeed, $Q(Z_r^{+\infty})$ contains the endpoints of the ascending diagonal of $Q$,
and $Q(Z_r^{-\infty})$ contains the endpoints of the descending diagonal.

Figure 2 shows a double ascending Zigzag and its prototype.

\centerline{\includegraphics[scale=1]{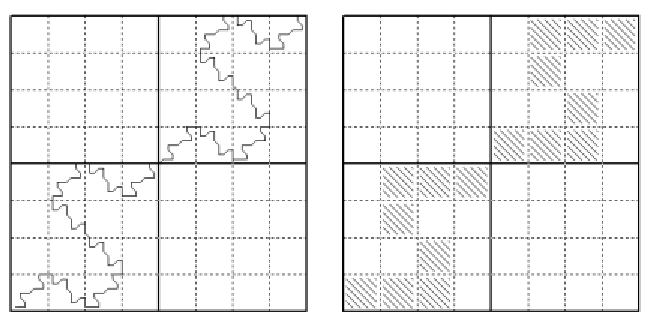}}
\centerline{Double Zigzag and its prototype.}

We now describe the fractal and block structure of the $r$-Zigzags.
An order-$n$ cell of $Q(Z_r^{\pm\infty})$ is its intersection with a cell of its $n$-fold prototype. By construction, every cell
of an $r$-Zigzag is homothetic to an ascending or descending $r$-Zigzag inscribed in its circumscribed square $[F]$; its diagonal type is determined by the parity of the corresponding prototype cell.

Order-$n$ cells of a multiple Zigzag belong to the same block if their corresponding cells in the $n$-fold prototype belong to the same block.
A block box of order $n$ is the square circumscribed about the union of eight order-$n$ cells.
forming a block.

 Lemmas \ref{zig-frac-odd} and \ref{zig-side}, \ref{frac-intersect} and \ref{zig-sosed} remain valid for Zigzags of any
The reader is invited to verify that Lemmas \ref{zig-frac-odd}, \ref{zig-side}, \ref{frac-intersect}, and \ref{zig-sosed} remain valid for Zigzags of any multiplicity.
The following observation may be helpful:
 if two cells of a multiple Zigzag of the same order intersect, then their extreme blocks intersect.

\subsection{Parameterizations of the Zigzag.}

In this section, we use the $l_\infty$ metric on the plane.
We call the Zigzag inscribed in $[0,1]^2$ the \emph{unit Zigzag}.
The ascending and descending unit Zigzags are isometric in this metric,
and since the results of this section concern only metric properties,
it suffices to consider the ascending unit Zigzag.
We denote the unit ascending Zigzag of the multiplicity $r$ by $\z^+_r$.

\begin{lem}\label{sosed-dist} Let $F$ and $G$ be disjoint order-$n$ cells of an $r$-fold unit Zigzag.
Then the distance between them is at least
$2^{-n(r+1)-1}$.
\end{lem}
\begin{proof}
If the circumscribed squares of the cells are disjoint, then the distance between the cells is at least the distance between these squares, which is at least $2^{-(r+1)n}$.
If the circumscribed squares intersect but the cells do not, the cells belong to the same block, whose central code is that of the Zigzag prototype.
By Lemma \ref{proto1}, the circumscribed squares of the cells have the same parity. Hence, Lemma \ref{proto2} implies that
the Zigzag prototypes inscribed in them are separated by at least half the side length of the squares. Since the cells lie inside these prototypes,
the distance between the cells is at least $2^{-(r+1)n-1}$.
\end{proof}

Let $Q(Z_r^{+n})$ denote the square chain whose initial point is the lower-left corner of $Q$.
For any point $x$ of the Zigzag $Q(Z_r^{+\infty})$,
let $Q_n(x)$ be the cell of the chain $Q(Z_r^{+n})$ containing $x$ that is farthest, in chain distance, from
the lowest extreme cell. Let $\nu_n(x)$ be the chain distance from $Q_n(x)$ to the lowest
cell of $Q(Z_r^{+n})$.
Then the following recurrence relation holds:
\begin{equation}\label{rec-nu}
 \nu_n(x)=2^{r+2}\nu_{n-1}(x)+\nu_n(x)\bmod2^{r+2}.
 \end{equation}

Define a map from the Zigzag $Q(Z_r^{+\infty})$ to the unit interval by
\begin{equation*}
 \nu(x)=\sum_{i=1}^{\infty} 2^{-i(r+2)} (\nu_i(x)\bmod 2^{r+2})
\end{equation*}

By virtue of \eqref{rec-nu}, we get

\begin{equation}\label{nu-lim}
 \nu(x)=\lim_{n\to\infty}
 \frac{\nu_n(x)}{2^{n(r+2)}}
\end{equation}

\begin{lem} For any points $x,y$ of the unit Zigzag $\z_r^+$, the following inequalities hold in the $l_\infty$ metric:
\label{hold-lem}    $$
2^{(r+2)^2+r+1}\ge \frac{|\nu(x)-\nu(y)|^{r+1}}{|x-y|^{r+2}}\ge 2^{-(r+2)^2}
$$
\end{lem}

\begin{proof}
Choose the smallest $n$ such that $|\nu_n(x)-\nu_n(y)|>1$.
Then $\nu(x)$ satisfies the inequalities
\begin{equation}\label{nu-ineq}
 2\cdot 2^{-(n-1)(r+2)} \ge |\nu(x)-\nu(y)| \ge 2^{-n(r+2)}
\end{equation}

In this case, the cells $Q_n(x)$ and $Q_n(y)$ are not neighbors, whereas
the order-$(n-1)$ cells $Q_{n-1}(x)$ and $Q_{n-1}(y)$ containing them are neighbors or coincide.
It follows from Lemma \ref{sosed-dist} that the $l_\infty$ distance between $x$ and $y$ satisfies
\begin{equation}\label{dist-ineq}
 2^{-n(r+1)-1} \le |x-y|\le 2\cdot 2^{-(r+1)(n-1)}
\end{equation}
Raising the inequalities in \eqref{nu-ineq} to the power $r+1$ and those in \eqref{dist-ineq} to the power $r+2$
gives the corresponding inequalities for the powered quantities.
The ratio of the left-hand sides is a power of two:
\begin{multline*}
-(n-1)(r+2)(r+1)+(r+1)+n(r+1)(r+2)+(r+2)=\\(r+1)(r+2)+(r+1)+(r+2)=(r+2)^2+r+1
\end{multline*}

The ratio of the right-hand sides is $2$ raised to the power
\begin{equation*}
-n(r+2)(r+1)-(r+2)+(r+2)(r+1)(n-1)=-(r+2)(r+1)-(r+2)=-(r+2)^2
\end{equation*}

The ratio of the middle terms lies between the ratios of the left- and right-hand sides.
\end{proof}

The map $\nu(x)$ continuously maps the Zigzag $\z_r^+$ onto the unit interval. Indeed,
the extreme points of the Zigzag map to the endpoints of the interval. The Zigzag is connected because every prototype is connected,
and the intersection of nested compact connected sets is connected.
Thus, the image of the Zigzag under $\nu$ is connected and therefore equals the unit interval.
The map $\nu$ is injective by Lemma \ref{hold-lem}; hence its inverse is continuous and has H\"older exponent
$\frac{r+1}{r+2}$ by the same lemma.

We denote the inverse of $\nu$, a map from the unit interval to the unit Zigzag $\z_r^+$, by
$\zeta_r^+(t)$ and call it the \emph{ascending structural parametrization} of $\z_r^+$.
By Lemma \ref{hold-lem}, this parametrization satisfies the following inequality for all $t,t\prime\in[0,1]$:
\begin{equation}\label{zig-param-ineq}
 |\zeta_r^+(t)-\zeta_r^+(t')|\le 2^{(r+2)}|t-t'|^{(r+1)/(r+2)}
\end{equation}

The discussion above can be summarized in the following theorem.
\begin{theorem} There is a H\"older-continuous homeomorphism between the unit interval and the unit $r$-fold Zigzag with H\"older exponent $(r+1)/(r+2)$; its inverse is H\"older-continuous with exponent $(r+2)/(r+1)$.
 \label{th-zigzag2}
\end{theorem}
\section{Squiggle}

 \subsection{Combinatorics.}
We call the union of an ascending and a descending Zigzag of the same multiplicity, arranged so that
their circumscribed squares form a Cartesian rectangle $[\w]$, a \emph{Squiggle}. The lower half of the rectangle is the circumscribed square of the ascending Zigzag,
and the upper half is the circumscribed square of the descending Zigzag.
The multiplicity of the Squiggle is that of its two Zigzags. A Squiggle of multiplicity $r$ is called an
$r$-Squiggle and is denoted by $\w_r$.

\paragraph{Cells.}
The order-$n$ cells of a Squiggle, for $n>0$, are the order-$n$ cells that make up its Zigzags.
The two Zigzags forming a Squiggle are called its \emph{floors} and are regarded as order-zero cells.
The order of the cell $F$ is denoted by $\ord F$.

The following Boolean matrix represents the first-order cells of a single Squiggle
\begin{equation}\label{matr}
 \begin{matrix}
  1 & 0 & 0 & 0 \\
  1 & 0 & 1 & 1 \\
  1 & 1 & 0 & 1 \\
  0 & 0 & 0 & 1 \\
  0 & 1 & 1 & 1 \\
  0 & 1 & 0 & 0 \\
  0 & 0 & 1 & 0 \\
  1 & 1 & 1 & 0
 \end{matrix}
\end{equation}
The entries equal to one mark the positions of cells in the rectangle $[\w_1]$.

The horizontal and vertical positions of the circumscribed square of a cell $F$ in the circumscribed rectangle $[\w]$
are denoted by $\Col F$ and $\Row F$, respectively.
We use $\col F$ and $\row F$ for the positions of $F$ in the circumscribed rectangle $[\pr F]$ of its parent cell.
Thus, for every positive-order cell $F$
of the $r$-Squiggle, the following relations hold:
\begin{equation}\label{recur}
 \Row F=\row F+2^{r+1}\Row \pr F \qquad \Col F=\col F+2^{r+1}\Col \pr F
\end{equation}

For order-zero cells, we set $\col F=\Col F$ and $\row F=\Row F$.
With this convention, the positions of any cell of an $r$-Squiggle can be expressed as follows:

\begin{equation}\label{glob-loc}
 \Row F=\sum_{i=0}^{\ord F} 2^{i(r+1)} \row \pr^i F \qquad \Col F=\sum_{i=0}^{\ord F} 2^{i(r+1)} \col \pr^i F
\end{equation}

We are interested in the residues of cell positions modulo 2, 4, and 8.
We call these the binary, quaternary, and octal abscissas and ordinates of the cells, respectively.
Note that a cell has the same quaternary coordinates relative to the entire Squiggle and relative
to its parent cell.
We denote the quaternary abscissa and ordinate of $F$ by $\col_4 F$ and $\row_4 F$, respectively.
Its binary coordinates are denoted by $\col_2 F$ and $\row_2 F$; their Boolean sum (modulo 2),
denoted by $\odd F$, is called the \emph{parity} of the cell.

\begin{lem}\label{frac-zag-odd}
Every even cell of the Squiggle is an ascending Zigzag, and every odd cell is a descending Zigzag.
 \end{lem}
\begin{proof}
Indeed, the claim holds for order-zero cells because the lower floor is even and ascending,
whereas the upper floor is odd and descending.
For cells of positive order, the claim follows from Lemma \ref{zig-frac-odd} for Zigzags.
\end{proof}

\paragraph{Columns.}
A set of cells in a Squiggle that have the same order and horizontal projection
is called a \emph{column}. We denote the column containing a cell $F$ by $C(F)$.

The order of a column $C$, denoted by $\ord C$, is the order of its cells.
A column of order zero has two rows.

For a column $C$, let $\pr C$ denote the parent column,
consisting of the parent cells $\pr F$ of all cells $F$ in $C$.

For a column $C$, $\Col C$ (respectively, $\col C$) denotes the horizontal position of its cells in the whole Squiggle (respectively, in their parent cells).
of its constituent cells.

The parity of a column is its binary abscissa, denoted by $\odd C$.
By Lemma \ref{frac-zag-odd},
in an even column each ascending cell has even vertical position and each descending cell has odd vertical position; in an odd column the reverse is true.

\begin{lem}\label{zag-tip}
 If $C(F)=C(F')$ and $\row_4 F=\row_4 F'$, then $\odd\pr F=\odd \pr F'$ and
$$\Row F=\Row F'\bmod 2^{r+2}.$$
\end{lem}
 \begin{proof}
Consider the block boxes of these cells. Equality of the quaternary ordinates of $F$ and $F\prime$ implies equality of their ordinates
relative to their block boxes. Hence the blocks have the same diagonal type, as do the parent cells, and the parent cells have the same parity.
Consequently, $\row F=\row F\prime$.

Since
equality of the parent-cell parities implies $\Row F=\Row F\prime\bmod 2^{r+2}$.
 \end{proof}

 \begin{lem}\label{stolb4}
Every column of first-order cells in a Squiggle of any multiplicity contains exactly four first-order cells with
distinct quaternary ordinates.
\end{lem}
\begin{proof}
The prototype of a multiple Zigzag consists of several Zigzag blocks arranged along the diagonal of the circumscribed square.
Each column of the multiple Zigzag intersects exactly one block box. The diagonal type of the block
matches that of the Zigzag. Thus, a column of the Squiggle intersects its upper floor in a block of descending type and its lower floor in a block of ascending type.

The lemma therefore holds for a Squiggle of any multiplicity if it holds for a single Squiggle,
as is clear from matrix \eqref{matr}.
\end{proof}

 \begin{lem}\label{odd=even}
Every column of an order-$n$ Squiggle contains $2^n$ even cells and $2^n$ odd cells.
\end{lem}
\begin{proof}
We prove the claim by induction on the order of the column. The base case is an order-zero column,
which consists of two cells of opposite parity.
Assume the claim holds for columns of order $n$, and consider a column $C$ of order $n+1$.
By the induction hypothesis, the parent column $\pr C$ contains $2^n$ pairs of cells of opposite parity.
Each cell of $\pr C$ is crossed by one block of the same type as that cell.
Thus, column $C$ intersects $2^n$ ascending blocks and $2^n$ descending blocks,
each made up of order-$n$ cells. By Lemma \ref{stolb4}, the intersection of the column with each ascending-descending block pair
consists of four cells with distinct quaternary ordinates.
Each such pair therefore contains cells of opposite parity.
\end{proof}

\paragraph{Blocks}

We denote the block containing a Squiggle cell $F$ by $B(F)$.
The positions of a block are the positions of its circumscribed rectangle $[B(F)]$ in the circumscribed rectangle of the Squiggle.
The block positions are the integer parts of the corresponding positions of the cell $F$ divided by four.
The parity of a block is denoted by $\odd B(F)=\Col B(F)+\Row B(F)$. For a single Squiggle,
$[B(F)]=[\pr F]$, so $\col B(F)=\row B(F)=0$ for every cell.

\begin{lem}\label{bloc-odd} For every positive-order cell $F$ of a Squiggle, $\odd B(F)=\odd\pr F$.

\end{lem}
\begin{proof}
The block and its parent cell have the same diagonal type.
The diagonal type of a cell is determined by its parity: even cells have even blocks because
they lie on an ascending diagonal, whereas odd cells lie on a descending diagonal and therefore have odd blocks.
\end{proof}

\begin{lem} For an $r$-Squiggle, the horizontal and vertical positions of the block $B(F)$ within its parent cell $\pr F$ satisfy the following relations:
$\col B(F)=\row B(F)$ if $B(F)$ is even, and $\col B(F)+\row B(F)=2^{r-1}-1$ if it is odd.
\label{blok-diag}
\end{lem}
\begin{proof}
 Indeed, if the block is even, then the parent cell is even, therefore ascending, and the block lies on the ascending diagonal.
 If the block is odd, then the parent cell is descending, and it lies on the descending diagonal.
\end{proof}

The horizontal and vertical positions of a block can therefore be expressed by a single formula valid for every multiplicity of the Squiggle:
\begin{equation}\label{block-pos}
 \col B(F)=\row B(F)+(2^{r-1}-1-2\row B(F))\odd B(F)
 \end{equation}

 \begin{lem}\label{blok-stolb}
Let two blocks of cells of the same order intersect the same column of a Squiggle.
Then the blocks have the same parity if and only if they have the same diagonal type.
\end{lem}
\begin{proof}
The parities of the parent cells of the blocks agree with the parities of the blocks.
The parent cells belong to the same column of the Squiggle, and cells of the same parity in a column
have the same diagonal type. The same is therefore true of the corresponding blocks.
 \end{proof}

\begin{lem} \label{8-spectr}
Every column of a Squiggle contains cells with all possible quaternary ordinates.
\end{lem}
\begin{proof}
For first-order columns, this follows from Lemma \ref{stolb4}. If a column has order greater than one,
its parent column contains cells of both types by Lemma \ref{odd=even}.
Choose cells $F$ and $F\prime$ in column $C$ with different parent cells.
Their blocks have different diagonal types and parities by Lemma \ref{blok-stolb}.
The quaternary positions within these blocks are the same as those of first-order cells of a single Squiggle,
so Lemma \ref{stolb4} implies that
$C$ intersects $B(F)\cup B(F\prime)$ in four cells with distinct quaternary ordinates.
\end{proof}

 \subsection{Arithmetic.}

A four-element set of integers with pairwise distinct residues modulo four is called a \emph{quartet}.

 \begin{lem}
\label{arif1} Let $(S_0,\dots,S_r)$ be a set of quartets, let $(b_0,\dots,b_r)$ be a Boolean vector, and let $Y$ be an integer with parity $b_0$. Then the equation
 \begin{equation}\label{equat}
  \sum_{i=0}^{r} 2^i x_i=Y \bmod 2^{r+2}
 \end{equation}
 has a unique solution provided $x_i\in S_i$, $x_i\bmod2=b_i$ for any $i\in [0,r]$.
 \end{lem}
 \begin{proof}
Set $X_k=\sum_{i=0}^{k}x_i2^i$.
We determine the values of $x_k$ inductively so that the congruences
$X_k\bmod 2^{k+1}=Y\bmod 2^{k+1}$ hold, as required by \eqref{equat}.
The base case $k=0$ follows from the lemma's assumption that $Y\bmod 2=b_0$.

Assume that $x_k$ has been uniquely determined for every $k<n<r$, so that $X_k\bmod 2^{k+1}=Y\bmod 2^{k+1}$ for all $k<n$.
Consider the congruence $X_{n+1}=2^{n+1}x_{n+1}+2^nx_n+X_{n-1}=Y\bmod 2^{n+2}$, which must hold.
It follows that $[x_n/2]=b_{n+1}-(X_{n-1}-Y)/2^{n+1}\bmod 2$, so the integer part of $x_n/2$ is uniquely determined.
Consequently, $x_n\bmod 4=b_n+2[x_n/2]$ is uniquely determined, and hence so is $x_n$.

 Consider the last step $$X_{r-1}+2^r x_r=Y\bmod 2^{r+2}. $$
That is, $(X_{r-1}-Y)/2^r+x_r=0\bmod 4$, which uniquely determines $x_r$.

 \end{proof}

\begin{lem} Given positive-order columns $C_0,\dots,C_r$ of an $r$-Squiggle, a Boolean vector $\beta=(b_0,\dots,b_r)$, and an integer $Y$,
with parity $b_0+\odd C_0$, there are cells $F_i\in C_i$ with parity vector $(\odd F_0,\dots,\odd F_r)=\beta$ such that

\begin{equation}\label{transit-eq}
 2^{r+1}\row_2\pr F_0+\sum_{i=0}^{r} 2^i\row F_i=Y \bmod 2^{r+2}
\end{equation}
The values of $\row F_i$ are uniquely determined.
\label{transit}
\end{lem}
\begin{proof}
Set $S_0=\{\row_4 F+2^{r+1}\row_2\pr F\mid F\in C_0\}$ and $S_i=\{\row F\mid F\in C_i\}$ for $i>0$.
Let $b\prime_i=b_i+\odd C_i$.
By Lemma \ref{8-spectr}, each $S_i$ is a quartet. Applying Lemma \ref{arif1} with parity vector $(b\prime_0,\dots,b\prime_r)$ gives the required result.
\end{proof}

\begin{lem}\label{frac-seq}
Let $C_0,\dots,C_r$ be columns of $r$-Squiggles of order $n\ge0$, let $(b_0,\dots,b_r)$ be a Boolean vector, and let $Y$ be an integer with parity $b_0+\odd C_0$.
Then there is a unique choice of cells $F_i\in C_i$ with $\odd F_i=b_i$ for which the following equality holds:

 \begin{equation*}
  \sum_{j=0}^{r} 2^j \Row F_j =Y \bmod 2^{n(r+1)+1}
 \end{equation*}
\end{lem}
 \begin{proof}
We prove the statement by induction on the order of the columns.
The case $n=0$ is left to the reader.
 Let $n=1$. Since $\{\Row F\mid F\in C\}$
For every column $C$, the set $\{\Row F\mid F\in C\}$ is a quartet by Lemma \ref{8-spectr};
thus, Lemma \ref{arif1} gives cells $F_i\in C_i$ of the prescribed parities for which
 \begin{equation}\label{seq1}
 Y= \sum_{i=0}^{r} 2^i \Row F_i \bmod 2^{r+2}
 \end{equation}
The values of $\Row F_i$ are therefore uniquely determined, and hence the cells $F_i$ are unique.
 Thus, the base of induction is established.

Assume that the claim holds for columns of order $n\ge1$, and consider columns $C_i$ of order $n+1$.

By Lemma \ref{transit}, there are cells $F_i\in C_i$ of the prescribed parities such that
\begin{equation}\label{step-n}
2^{r+1}\row_2\pr F_0+\sum_{i=0}^{r} 2^i\row F_i=Y \bmod 2^{r+2}
 \end{equation}
The values of $\row F_i$ are uniquely determined.

Since $Y_1=2^{-r-1}(Y-\sum_{i=0}^{r}2^i\row F_i)$ has the same parity as $\row_2\pr F_0$,
the induction hypothesis applied to the parent columns $\pr C_i$ uniquely determines cells $G_i\in\pr C_i$ with $\odd G_i=\odd\pr F_i$ such that
 \begin{equation}\label{step=n1}
 \sum_{i=0}^{r} 2^i \Row G_i =Y_1 \bmod 2^{n(r+1)+1}
 \end{equation}
The cells $G_i$ have the same types as $\pr F_i$, so they contain subcells $F_i\prime\subset G_i$ of order $n+1$ for which
$C(F_i\prime)=C(F_i)$ and $\row F_i=\row F_i\prime$; the subcells are uniquely determined by their positions within their parent cells.
By Lemma \ref{zag-tip},
$\Row F_i\prime=\Row F_i\bmod 2^{r+2}$, so the required congruence holds for the cells $F_i\prime$.

Replacing $G_i$ by $\pr F_i$ in \eqref{step=n1} and multiplying by $2^{r+1}$ gives
$$
\sum_{i=0}^{r}2^i ( 2^{r+1}\Row \pr F_i)
=
 Y-\sum_{i=0}^{r} 2^i\row F_i \bmod 2^{(n+1)(r+1)+1}
$$
Moving the sum on the right-hand side to the left yields the required equality for columns of order $n+1$:
$$
\sum_{i=0}^{r} 2^i ( 2^{r+1}\Row \pr F_i+\row F_i)= \sum_{i=0}^{r} 2^i \Row F_i = Y \bmod 2^{(n+1)(r+1)+1}
$$
\end{proof}

\begin{theorem}\label{th-razvert}
For any $x_0,\dots,x_r$ in the unit interval and any $Y\in[0,2]$,
there exist $y_0,\dots,y_r$ such that the following conditions hold:
\begin{enumerate}
 \item for any integer $i\in [0,r]$ point $(x_i,y_i)$ belongs to the Squiggle $\w_r$, inscribed in the rectangle $[0,1]\times[0,2]$;
 \item $\sum_{j=0}^{r} 2^j y_j =Y \bmod 2$
 \end{enumerate}

\end{theorem}
\begin{proof}
The set of values of $Y$ for which the theorem's conditions hold is compact, so it suffices to prove that it is dense.
The set of real numbers that become integers after multiplication by some power of two is dense in $\R$;
therefore, it suffices to prove the theorem for such values of $Y$.

Let $2^nY$ be an integer for some natural number $n$.
For each $x_j$, choose an order-$n$ column $C_n^j$ of the Squiggle $\w_r$ whose base contains $x_j$.
By Lemma \ref{frac-seq}, there are cells $F_j\in C_n^j$ for which
 \begin{equation}\label{sum-eq1}
  \sum_{j=0}^{r} 2^j \Row F_j =2^{n(r+1)}Y \bmod 2^{n(r+1)+1}
 \end{equation}
Let $y_j(n)$ be the ordinate of the lower-left corner of $F_j$, and let $x_j(n)$ be its abscissa.
Dividing both sides of \eqref{sum-eq1} by $2^{n(r+1)}$ and using $y_j(n)=2^{-n(r+1)}\Row F_j$,
we obtain
 \begin{equation*}
  Y= \sum_{j=0}^{r} 2^j y_j(n) \bmod 2
 \end{equation*}

The second condition is therefore satisfied. The resulting points $(x_j(n),y_j(n))$ may not lie on the Squiggle, and their abscissas
may not equal the prescribed values. However, compactness of the Squiggle allows us to choose a sequence of natural numbers
$n_k$ such that, for every $j$, the limits $\lim_{k\to\infty}x_j(n_k)$ and $\lim_{k\to\infty}y_j(n_k)$ exist;
these limits satisfy all the required conditions.

 \end{proof}

Define the parametrization of the Squiggle, $\zeta_r\colon[0,1]\to\w_r\subset\C$, by
\begin{equation}\label{zag-par}
 \zeta_r(t)=\begin{cases}
    \zeta_r^+(2t), & \mbox{if } t\le\frac12 \\
    (1+\i)+\i\zeta_r^+(2t-1), & \mbox{if } t\ge\frac12.
    \end{cases}
\end{equation}

\begin{lem}\label{zag-coef} For all $t,t\prime\in[0,1]$, the following inequality holds in the $l_\infty$ metric:
\begin{equation}\label{zag-dil}
 |\zeta_r(t)-\zeta_r(t')|\le 2^{(r+3)-1/(r+2)} |t-t'|^{(r+1)/(r+2)}
\end{equation}
\end{lem}
\begin{proof}
Consider the third cell $F_3$ and fourth cell $F_4$ of the first-order ascending Zigzag $\z_r^+$. Their
circumscribed squares lie one above the other. The third cell is an ascending Zigzag and the fourth is a descending one.
Their union is therefore mapped onto a Squiggle by a homothety $H$ with scale factor $2^{r+1}$.
The parametrization $\zeta_r^+$ maps the interval $[2\cdot2^{-(r+2)},4\cdot2^{-(r+2)}]$ onto $F_3\cup F_4$.
Let $\tau(t)=2^{-(r+2)}(2t+2)$. Then $H\circ\zeta_r^+\circ\tau=\zeta_r$.
It follows that
 \begin{multline*}
  |\zeta_r(t)-\zeta_r(t')|\le 2^{r+1} |\zeta_r^+(\tau(t))-\zeta_r^+(\tau(t'))|
 \le 2^{r+1}(2^{(r+2)}|\tau(t)-\tau(t')|^{(r+1)/(r+2)})\le \\
 2^{(2r+3)}
 (2^{-(r+1)^2/(r+2)} |t-t'|^{(r+1)/(r+2)})=2^{(r+3)-1/(r+2)} |t-t'|^{(r+1)/(r+2)}
  \end{multline*}

\end{proof}

\begin{theorem}\label{th-zag} For every integer $n>1$, there exists a surjective map $\alpha_n\colon[0,1]^n\to[0,1]^{n+1}$
that satisfies the following H\"older inequality for all $x,y\in[0,1]^n$:
$$ |\alpha_n(x)-\alpha_n(y)|\le (2^n-1)2^{n+2-1/(n+1)} |x-y|^{n/(n+1)} $$
\end{theorem}
\begin{proof}
The desired map is defined using the parametrization of an $(n-1)$-fold Squiggle inscribed in a rectangle
 $[0,1]\times[0,2]\subset \mathbb{C}$
 by the following formula:

 \begin{equation*}
  \alpha_n(t_1,\dots, t_n)=(\re \zeta_{n-1}(t_1),\dots,\re \zeta_{n-1}(t_n),\left | \left(\sum_{k=1}^n 2^{k-1} \im \zeta_{n-1}(t_k)\right) \bmod 2 -1\right |)
 \end{equation*}

Indeed, the $\alpha_n$ map is a superposition of the $n$-th Cartesian degree of the H\"older-continuous map $\zeta_r(t)$
and the Lipschitz map with the stretching coefficient $\sum_{k=1}^n2^{k-1}=(2^n-1)$.
 And since the stretching coefficient in the metric $l_\infty$ and the H\"older exponent of the Cartesian degree of mapping are the same as
 for the mapping itself, Lemma \ref{zag-coef} provides the required inequality.

 The surjectivity of the $\alpha_n$ mapping is provided by the theorem \ref{th-razvert}.
 \end{proof}

\section{Fractal Cross}

For a square Boolean matrix $M=\{M(i,j)\mid i,j\in [0,n-1]\}$ of size $n\times n$ and a Cartesian square $S$, we define
$$M(S)=\{F\in S/(n\times n)\mid M(\col F,\row F)=1\}.$$
We will call the elements of the set $M(S)$ $M$-\emph{selected} cells, and their union an $M$-selected subset of the square $S$.

The superposition of matrix selections of the Cartesian square $ M(N(S))=\{M(F)\mid F\in N(S)\}$ is a matrix selection with a matrix equal to the tensor product $M\otimes N$.

 For any natural number $r$, we define the Boolean matrix $\x_r$ of size $2^{r+1} \times 2^{r+1}$ such that
 $$ \x_r(i,j)=\begin{cases}
    1, & \mbox{if $i=j$ or $i+j=2^{r+1}-1$} \\
    0, & \mbox{in all other cases}.
    \end{cases}$$

A partition cell is called \emph{diagonal} if its intersection with any diagonal of the square $S$ is a diagonal of the cell.

For a given natural $r$ and a square $S$ (not necessarily Cartesian), is defined its
$2^{r+1}$-\emph{diagonalization} as the set of diagonal cells
of the regular $2^{r+1}\times 2^{r+1}$ partition of the square $S$.
 As it is easy to see, the $2^{r+1}$--diagonalization of the square $S$ coincides with its $\x_r$-partition.

 Thus, the $2^{r+1}$--diagonalization consists of $2^{r+2}$ square cells homothetic to the square $S$ with a similarity coefficient of $2^{-(r+1)}$.

 Note that the body $\x_0 (S)$ coincides with the original square, whereas the cells of the 2-diagonalization represent four quarters of the original square. $\x_0 (S)=S/(2\times 2)$.

 By induction, we define an $n$-fold $2^{r+1}$--diagonalization, the totality of which cells form a single
$2^{r+1}$--diagonalization of cells $(n-1)$-multiple $2^{r+1}$--diagonalization.
The body of the $n$-fold diagonalization of the square $S$ is denoted by $\cup\x_r^n S$.
Thus, the cells of the $n$-fold $2^{r+1}$--diagonalization of the unit square are squares with side length $2^{-(r+1)n}$.
The bodies of the $n$-fold $2^{r+1}$--diagonalizations of the square $S$ form a sequence that decreases as the multiplicity increases; the intersection of this sequence is denoted by $\x_r^\infty S$ and is called the \emph{fractal $2^{r+1}$--adic Cross} inscribed in the square $S$.

The matrix of the $n$-fold $2^{r+1}$--diagonalization of a Cartesian square is
 tensor product of $n$ matrices $\x_r$.

Note that $\x_0^n S=S/(2^n\times 2^n)$ and $\x_0^\infty S=S$.
\bigskip

\centerline{\includegraphics[scale=0.8]{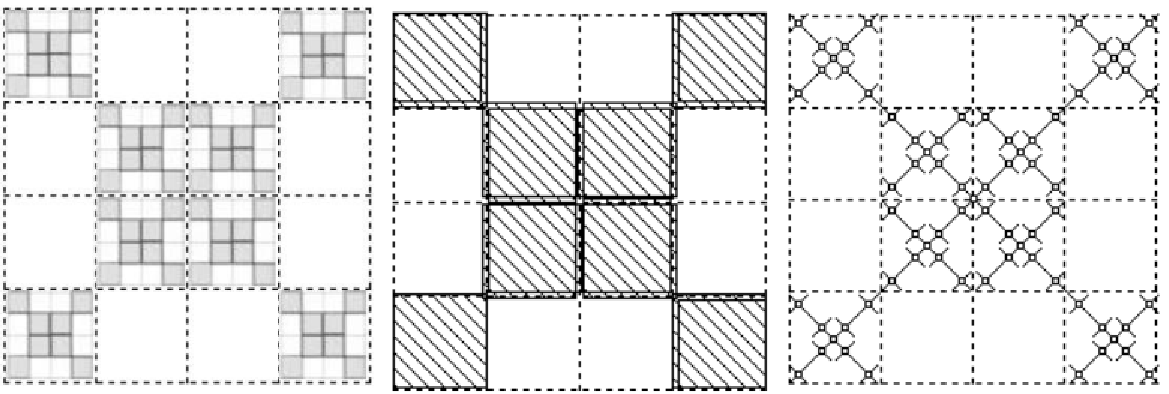}}
 \centerline{Figure 3}
\bigskip
Figure 1 shows the $2^{r+1}$--diagonalization for $r=1$ in the center. On the right is the fractal Cross $\x_1^\infty S$, and on the left is the double diagonalization $\x^2_1 S$.
\bigskip

\centerline{\includegraphics[scale=1]{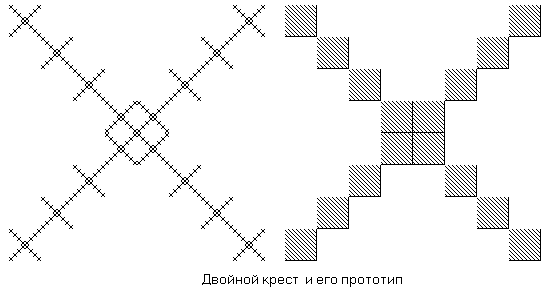}}
 \centerline{Figure 4}

\bigskip
 Figure 2 shows the fractal Cross $\x^\infty_3 S$ on the left, and $2^{r+1}$--diagonalization for $r=3$ on the right.

\begin{lem} For a square $S$ and any natural number $r$, the inscribed fractal Cross has the following properties:
 \label{prop-crest}
 \begin{enumerate}
      \item $\x_r^\infty S$ is entirely contained in the square $S$ and contains both diagonals of this square.
      \item Any isometry of the square $S$ maps $\x_r^\infty S$ onto itself.
      \item $\x_r^\infty S$ consists of $2^{r+2}$ similar cells with a coefficient of $2^{-(r+1)}$.
  \end{enumerate}
 \end{lem}
 \begin{proof}
 The first and second properties are obvious.
 Let us consider the third property.
 Let $F$ be a cell of the $2^{r+1}$-diagonalization of the square $S$. Then, a similarity transformation with a scaling factor of $2^{r+1}$,
 which maps $F$ to $S$, maps $\x_r^n F$ to $\x_r^n S$ for any $n$, including infinite ones.
  On the other hand, the union of $\x_r^n F$ over all cells of a single diagonalization yields an $(n+1)$-fold diagonalization of the square $S$. This proves the self-similarity of $\x_r^\infty S$.
 \end{proof}

\section{Parametrization of the fractal Cross}
In this section, we will construct a mapping of a segment onto the Cross $\x_r^\infty S$ with a H\"older exponent of $\frac{1+r}{2+r}$.
Moreover, for $r=0$, the resulting curve will coincide with the well-known Sierpinski curve.

We will assume that the plane in which we carry out the constructions is complex, so that a point with coordinates $(x,y)$ corresponds to the complex number $x+\i y$. We will denote by $S$ the square with diagonal $[-1,1]$.
 The diagonals of this square induce a partition of all the diagonalization cells of $\x_r S$ into pairs of isosceles right triangles, which we will call the \emph{triangular cells} of the diagonalization of $\x_r S$.
 The set of triangular diagonalization cells $\x_r S$ will be called the $2^{r+1}$-\emph{triangulation} of the square $S$.

 All triangulation cells have hypotenuses parallel to one of the coordinate axes.
 The triangular cells are divided into four types depending on the relative position of their right-angled vertex and hypotenuse:
 \begin{itemize}
 \item \emph{upper} A triangle with a horizontal hypotenuse located above its hypotenuse is called:
 \item \emph{lower} — a triangle with a horizontal hypotenuse located below its hypotenuse;
 \item \emph{right} — a triangle with a vertical hypotenuse located to the right of its hypotenuse;
 \item \emph{left} — a triangle with a vertical hypotenuse located to the left of its hypotenuse.
 \end{itemize}

We consider all triangles to be positively oriented, with an orientation induced by the complex plane, that is, the sides of the triangle are oriented in such a way that when moving along a side from the beginning to the end, the triangle remains on the left.
\emph{ The beginning} of an oriented right triangle will be called the initial vertex of its hypotenuse, that is, such an acute vertex of the triangle that when moving from this vertex along the hypotenuse, the triangle remains on the left.
The end of the right triangle is called the other acute vertex.

The highlighted vertices are the left vertex in the upper triangle, the right vertex in the lower triangle, the upper vertex in the right triangle, and the lower vertex in the left triangle.
Multiplication by $\i$ maps the upper triangle to the left-hand position, the left triangle to the bottom, the lower triangle to the right, and the right triangle to the top, carrying the distinguished vertex to the distinguished vertex.

\begin{lem}\label{podobie}
For any isosceles right triangle $T$, there is an orientation-preserving similarity
mapping $T$ to the upper half of $S$ and taking its initial vertex to $-1$.
 \end{lem}
 \begin{proof}
 The proof is left to the reader.
 \end{proof}
The triangulation of the triangle $T$ induced by this similarity is called its $2^{r+1}$--triangulation and is denoted by $\x_r T$. We denote the upper half of the square $S$ by $S^+$.
Thus, $\x_r S^+$ is a triangle with vertices $-1,\i,1$, divided into $2^{r+2}$ triangular cells.
The triangular cells of $\x_r S^+$ have the following properties:
 \begin{itemize}
 \item the terminal vertex of any cell is the initial vertex for some cell;
 \item the terminal vertex of any cell is the initial vertex for no more than two cells;
 \item if the terminal vertex of some cell is the initial vertex for two cells, then the hypotenuses of these cells are mutually orthogonal.;
 \end{itemize}
The verification of these properties is left to the reader (see Figure 5).

\centerline{\includegraphics[scale=0.8] {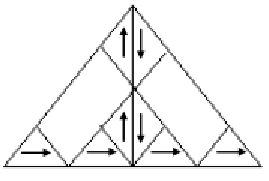}}
 \centerline{Figure 5}

For each triangular cell in $\x_r S^+$, define its successor according to the following rules:
\begin{enumerate}
\item If the terminal vertex of a triangle is the initial vertex of only one triangle, that triangle is the next one.
\item If the terminal vertex of a triangle $T$ is the initial vertex of two triangles, the next triangle is the one whose hypotenuse is orthogonal to that of $T$.
\end{enumerate}

A sequence of right triangles $T_1,T_2,\dots,T_n$ is called an \emph{admissible chain} if each $T_{k+1}$ follows $T_k$ according to these rules for $k<n$.
$T_1$ is the beginning of the chain, and $T_n$ is its end.

\begin{lem} The triangular cells of $\x_r S$ can be numbered uniquely to form an admissible chain of length $2^{r+3}$ that begins and ends at $-1$; the first half of the chain lies in $S^+$ and ends at $1$.
\label{punktir}
\end{lem}
\begin{proof}
The proof is left to the reader.
\end{proof}

For any isosceles right triangle $T$, define its $2^{r+1}$-triangulation $\x_r T$ as the partition into triangular cells induced by an orientation-preserving similarity from $T$ to the upper half of the square $S$.
This similarity maps the chain of triangular cells in the upper half of $S$ to an admissible chain of cells in the triangulation $\x_r T$, beginning and ending at the endpoints of the hypotenuse of $T$.

\begin{lem}\label{chain-decomp}
If $T_1,\dots,T_n$ is an admissible chain, replacing any triangle in the chain
with an admissible chain of cells from its $2^{r+1}$-triangulation yields another admissible chain.
\end{lem}
\begin{proof}
The proof is left to the reader.
\end{proof}

We denote by $\partial\x_r^nS$ the union of the boundaries of the triangular cells in the $n$-fold diagonalization of $S$.
The set $\partial\x_r^\infty S=\x_r^\infty \cap\cup_{n=1}^\infty\partial\x_r^n$.
This union is called the \emph{skeleton} of the Cross $\x_r^\infty S$.
\begin{lem}\label{skeleton} The complement of the skeleton of the Cross $\x_r^\infty S$ is dense in the Cross.
\end{lem}
\begin{proof}
Since similarities between Crosses preserve the skeleton, the density of the complement follows from the fact that the skeleton is not equal to the Cross.
\end{proof}

A map $f\colon X\to Y$ between topological spaces is called \emph{irreducible} if $f(X')\ne f(X)$ for every proper closed subset $X'\subset X$.

\begin{theorem} For every nonnegative integer $r$, there exists an irreducible H\"older-continuous map $\tau_r$
from the unit interval onto the fractal Cross $\x_r^\infty S$, with H\"older exponent $\frac{r+1}{r+2}$ and H\"older constant
$2^{(r+2)+(r+1)/(r+2)}$, where the Cross is inscribed in the square $S$.
 \label{th-serp}
\end{theorem}
\begin{proof}
Consider an admissible chain of triangular cells $T_1,\dots,T_{2^{r+3}}$ in the diagonalization $\x_r S$.
Replacing each triangular cell in the chain by its $2^{r+1}$-triangulation gives an admissible chain of cells in the second diagonalization $\x_r^2S$.
Repeating this process gives admissible chains of triangular cells in diagonalizations of every order.
Let $z_0,z_1,\dots,z_{2^{(r+2)n+1}}$ be the initial vertices of the triangular cells in such a chain.
Here, $n$ denotes the order of the diagonalization of $S$.
Define $\tau_r(k2^{-(r+2)n-1})=z_k$ for every integer $k\in[0,2^{(r+2)n+1}]$.
This is well-defined: if $k$ is divisible by $2^{r+2}$, then the subsequence $z_k$ with indices divisible by $2^{r+2}$ is exactly the sequence of initial vertices at the preceding subdivision level,
so the corresponding terms agree with those of the previously defined sequence.

We now prove H\"older continuity of $\tau_r$. Let
$t<t\prime\le1$. Then, for some integer $k$, the following inequalities hold:
 \begin{equation}\label{est-low}
 2^{-(r+2)k-1} \le t'-t <2^{-(r+2)(k-1)-1}
 \end{equation}
The points $\tau_r(t)$ and $\tau_r(t\prime)$ then lie either in the same cell of the $(k-1)$-fold diagonalization or in two neighboring cells.
The distance between their images is therefore at most twice the diameter of a cell in the $(k-1)$-fold diagonalization,
 that is,
 \begin{equation}\label{est-up}
 |\tau_r(t)-\tau_r(t')|\le 2^{-(r+1)(k-1)+1}.
 \end{equation}
Combining inequalities \eqref{est-low} and \eqref{est-up} gives
 Therefore
 \begin{equation}\label{ratio}
 \frac{|\tau_r(t)-\tau_r(t')|^{r+2}}{|t'-t|^{r+1}}\le \frac{2^{-(r+1)(r+2)(k-1)+(r+2)}}{2^{-(r+2)(r+1)k-(r+1)}}=2^{(r+2)^2+(r+1)}
 \end{equation}
Thus, $\tau_r$ is H\"older-continuous with exponent $\frac{r+1}{r+2}$ and constant $2^{(r+2)+(r+1)/(r+2)}$.
This map extends continuously from a dense subset to the entire interval without changing its H\"older exponent or constant.
 \eqref{ratio}.
The image of $\tau_r$ lies in the union of the diagonals of the hypotenuses of triangular cells, and hence in the fractal Cross.

We next prove that $\tau_r$ is irreducible.
By construction, if $[t2^{(r+2)n+1}]=[t\prime2^{(r+2)n+1}]$, then $\tau_r(t)$ and $\tau_r(t\prime)$ lie in the same triangular cell of the $n$-fold diagonalization $\x_r^nS$.
Conversely, if these integer parts differ, then $\tau_r(t)$ and $\tau_r(t\prime)$ lie in different triangular cells.
 $\x_r^n S$.
Distinct cells of $\x_r^nS$ intersect only along a common vertex or side.
Thus, if $\tau_r(t)$ is not in the skeleton of the Cross, it has a unique preimage under $\tau_r$.
By Lemma \ref{skeleton}, points with unique preimages are dense in both the image and the domain,
which proves irreducibility.
\end{proof}

By a similar argument, every point of the Cross skeleton other than a vertex of a triangular cell
has exactly two preimages under $\tau_r$.
Each vertex of a triangular cell has multiplicity at most four.

\paragraph{Approximating polylines.}

A \emph{hypotenuse polyline} consists of the oriented hypotenuses of the triangular cells in an $n$-fold diagonalization of a square.
All hypotenuses are horizontal or vertical, so their directions can be encoded by complex numbers
$\pm1,\pm\i$. Since all hypotenuses have the same length, their directions determine the displacement vector
from the initial to the terminal point of each hypotenuse, and hence to the initial point of the next triangle.
The \emph{chain code} of a hypotenuse polyline is the sequence of hypotenuse directions, with commas and the initial plus sign omitted.
Thus, the chain code of the $2^3$-triangulation of the upper half of $S$ is written as follows.
\begin{equation}\label{kod}
 ( 1+1+\i+\i-\i-\i+1+1)
\end{equation}

Given chain codes $A=(a_1,\dots,a_n)$ and $B=(b_1,\dots,b_m)$, a symbol $x\in\{\pm1,\pm\i\}$, and a natural number $k$,
we define the following operations:
\begin{description}
 \item[addition] $A+B=(a_1,\dots,a_n,b_1,\dots,b_m)$
\item[multiplication] $xA=(xa_1,xa_2,\dots, xa_n)$
\item[repetition] the concatenation of $k$ copies of $A$, denoted by $k\times A$;
\end{description}
With this notation, the chain code of the hypotenuse polyline for the $2^{r+1}$-triangulation of $S^+$ is
\begin{equation}\label{r-code}
 C_r= 2^{r}\times(1)+2^{r}\times(\i)+2^{r}\times(-\i)+2^{r}\times(1)
\end{equation}
The hypotenuse polyline of the $2^{r+1}$-triangulation of any upper triangle has the same chain code.
The code for the lower triangle is obtained by multiplying the upper-triangle code by $-1$;
for the left and right triangles, multiply it by $\i$ and $-\i$, respectively.

Let $C_r^n$ denote the chain code of the hypotenuse polyline of the $n$-fold $2^{r+1}$-triangulation of $S^+$. Then
the following recurrence relation holds for the hypotenuse polylines:
\begin{equation}\label{recur-lom}
 C_r^{n+1}= 2^{r}\times C_r^n +\i 2^{r}\times C_r^n-\i 2^{r}\times C_r^n+2^{r}\times C_r^n
\end{equation}

The image of each hypotenuse polyline lies in the fractal Cross, and the lengths of these polylines tend to infinity.
To use them to approximate the map in Theorem \ref{th-serp}, we can parametrize them by
the ratio of the length traversed to the total length of the polyline.

Every hypotenuse occurs twice in a hypotenuse polyline, since the path traverses each hypotenuse
in both directions. The centers of the square
cells are therefore points of multiplicity four for $\tau_r$.
Alternatively, $\tau_r$ can be approximated by central polylines whose vertices lie at the centroids
of the triangular cells. These polylines are interesting because they have no self-intersections, so
$\tau_r$ can be approximated by embeddings.
The curve $\tau_r$ is closed and, for $r=0$, coincides with the Sierpinski--Knopp curve \cite{Serp}.

\subsection{Fractal Cross arithmetic.}

Integers in the range $[0, 2^{r+1}-1]$ are called $r$-digits.

 The two $r$ digits $a$ and $b$ will be called $r$-\emph{complementary} if $a+b=2^{r+1}-1$.
 We will denote the digit $r$-complementary to the digit $a$ by $a^*$.

\begin{lem}\label{dig-pres} Every cell in the $n$-fold diagonalization $\x_r^n S$ of the unit square $S=[0,1]^2$
has a lower-left corner with coordinates of the form
\begin{equation}\label{vid}
 \sum_{k=1}^{n} (a_k 2^{-k(r+1)}, b_k2^{-k(r+1)}),
 \end{equation}
where $a_k$ and $b_k$ are $r$-digits that either coincide or are $r$-complementary. Conversely, every expression of the form \eqref{vid}
gives the lower-left corner of a cell of $\x_r^n S$.
 \end{lem}
\begin{proof}
For $n=1$, if the cell has a diagonal on the line $x=y$, then its lower-left corner lies on that diagonal, so $a_1=b_1$. If its diagonal lies on $x+y=1$, the sum of the coordinates of its lower-left corner
is one minus the side length of the cell; hence $a_1+b_1=2^{r+1}-1$.

Assume the claim holds for order $n$. To find the lower-left corners of cells in $\x_r^{n+1}S=\x_r^1(\x_r^n S)$, add to the coordinates
of the lower-left corner of each order-$n$ cell, scaled by $2^{-(r+1)n}$, the corresponding coordinates
of the lower-left corners in the first diagonalization $\x_r^1S$, which have the required form by the preceding argument.

The number of cells in $\x_r^nS$ equals the number of expressions of the form \eqref{vid}, and these expressions give distinct sums.
It follows that the sums are in one-to-one correspondence with the lower-left corners of the cells.

 \end{proof}

We denote the set of two signs $\{+,-\}$ by $\pm$.
For a digit $a$ and a sign $\sigma\in\pm$, define
 \begin{equation}\label{comp}
  a^{\sigma}=\left\{
    \begin{array}{ll}
    a, & \hbox{if $\sigma=+$;} \\
    a^*, & \hbox{if $\sigma=-$.}
\end{array}
   \right.
 \end{equation}

Let $\x_r^nS(x)$ denote the subset of $\bigcup\x_r^nS$ consisting of points with abscissa $x$.

\begin{lem} If $\x_r^nS(x)$, for $x\in[0,1]$, contains a point of the form $\sum_{k=1}^{n}a_k2^{-k(r+1)}$, then it also contains
every point of the form
 \begin{equation}\label{sign-vid}
 \sum_{k=1}^{n} a_k^\pm 2^{-k(r+1)},
 \end{equation}
 where the signs for different $k$ are chosen independently.
 \label{intersect-vert}
\end{lem}
\begin{proof}
Indeed, if a vertical line intersects a cell in the $2^{r+1}$-diagonalization whose lower side has ordinate $a2^{-(r+1)}$,
then it also intersects the cell whose lower side has ordinate $(2^{r+1}-a-1)2^{-(r+1)}$. Thus, the lemma holds for a single
$2^{r+1}$-diagonalization.
The general statement follows by induction on the order $n$.
\end{proof}

\begin{lem}\label{sum-res}
For any sequence of $r+1$ digits $(a_1,\dots,a_{r+1})$, distinct sums $\sum_{k=1}^{r+1}a_k^\pm2^{k-1}$ have distinct residues modulo $2^{r+1}$.
Every residue modulo $2^{r+1}$ is represented by exactly one sum of this form.
 \end{lem}
 \begin{proof}
 Assume that \begin{equation}\label{sum-res1}
     \sum_{k=1}^{r+1} a_k^{\sigma(k)} 2^{k-1}= \sum_{k=1}^{r+1} a_k^{\theta(k)} 2^{k-1}\mod 2^{r+1}
     \end{equation}
We prove by induction that $\sigma(k)=\theta(k)$.
First, $\sigma(1)=\theta(1)$ because $a_1^{\sigma(1)}$ and
$a_1^{\theta(1)}$ have the same parity by \eqref{sum-res1}, so they cannot be complementary with respect to $r$.

Assume that $\theta(k)=\sigma(k)$ for all $k\le m$. Then \eqref{sum-res1}
implies the following equality of partial sums:
 \begin{equation}\label{sum-res2}
 \sum_{k=1}^{m+1} a_k^{\sigma(k)} 2^{k-1}=\sum_{k=1}^{m+1} a_k^{\theta(k)} 2^{k-1} \mod 2^{m+1}
 \end{equation}
The partial sums can differ only in their final terms. By the induction hypothesis,
these terms must have the same parity, which implies
$\sigma(m+1)=\theta(m+1)$. Hence \eqref{sum-res1} implies that the sign sequences agree.

There are $2^{r+1}$ such sums, equal to the number of residues modulo $2^{r+1}$;
therefore, the sums are in one-to-one correspondence with those residues.
 \end{proof}

For positive real numbers $x,y$, define
 $$ x\mod y= x-[x/y] $$

\begin{lem}\label{perenos}
For any natural number $n$, any $x_1,x_2,\dots,x_{r+1}\in[0,1]$, and any $Y\in[0,1]$, there exist points $y_k\in\x_r^nS(x_k)$ such that
 $$|Y-\sum_{k=1}^{r+1} y_k 2^{k-1}|\mod 1 \le 2^{-(r+1)n}.$$
\end{lem}
 \begin{proof}
Choose digits $a_1,\dots,a_n$ such that $\left|Y-\sum_{k=1}^{n}a_k2^{-(r+1)k}\right|<2^{-n(r+1)}$.
It is enough to find points $y_k\in\x_r^nS(x_k)$ satisfying
\begin{equation}\label{p1}
 \sum_{k=1}^{r+1} y_k 2^{k-1}=\sum_{k=1}^{n} a_k2^{-(r+1)k} \mod 1
\end{equation}

Choose $y_k\in\x_r^nS(x_k)$ on the lower boundary of a cell of $\x_r^nS$.
Then $y_k=\sum_{j=1}^{n}y_{kj}2^{-(r+1)j}$ for suitable digits $y_{kj}$. By Lemma \ref{intersect-vert},
every point of the form $\sum_{j=1}^{n}y_{kj}^{\pm}2^{-(r+1)j}$ also belongs to $\x_r^nS(x_k)$.
Substituting $\sum_{j=1}^{n}y_{kj}^{\pm}2^{-(r+1)j}$ into \eqref{p1} and changing the order of summation gives
the equation
\begin{equation}\label{p2}
 \sum_{j=1}^{n} 2^{-(r+1)j} \sum_{k=1}^{r+1} 2^{k-1}y_{kj}^{\sigma(kj)} =\sum_{k=1}^{n} a_k2^{-(r+1)k} \mod 1,
\end{equation}
where the signs $\sigma(kj)$ are to be determined.
We determine the signs $\sigma(kj)$ successively, starting with $l=n$ and decreasing $l$, so that
the congruence
\begin{equation}\label{p3}
 \sum_{j=l}^{n} 2^{-(r+1)j} \sum_{k=1}^{r+1} 2^{k-1}y_{kj}^{\sigma(kj)} =\sum_{k=l}^{n} a_k2^{-(r+1)k} \mod 2^{-(l-1)(r+1)}
\end{equation}

By Lemma \ref{sum-res}, we can choose the signs $\sigma(kn)$ so that
\begin{equation}\label{p4}
 \sum_{k=1}^{r+1}2^{k-1} y_{kn}^{\sigma(kn)}=a_n \mod 2^{r+1}
\end{equation}
Multiplying this equality by $2^{-(r+1)n}$ gives \eqref{p3}, establishing the base case of the induction.
Assume the signs $\sigma(kj)$ have been chosen so that \eqref{p3} holds for some $l\le n$.
When $l$ is decreased by one, the following term is added to the left-hand side:
\begin{equation}\label{p44}
 2^{-(r+1)(l-1)}\sum_{k=1}^{r+1} 2^{k-1} y_{k(l-1)}^{\sigma(k,(l-1))}
\end{equation}
and $a_{l-1}2^{-(r+1)(l-1)}$ is added to the right-hand side. Since the difference between the two sides has the form
$m2^{-(r+1)(l-1)}$ for some integer $m$, the induction step is complete if we choose signs $\sigma(k,l-1)$ satisfying
\begin{equation}\label{p5}
 \sum_{k=1}^{r+1} 2^{k-1} y_{k(l-1)}^{\sigma(k,(l-1))}=a_{l-1} +m \mod 2^{r+1}
\end{equation}
Such signs exist by Lemma \ref{sum-res}.
 \end{proof}

\begin{lem}\label{lem-razvert}
For any $x_1,\dots,x_{r+1}$ in the unit interval and any $Y\in[0,1]$,
there exist $y_1,\dots,y_{r+1}$ such that the following conditions hold:
\begin{enumerate}
 \item $(x_i,y_i)\in \x^\infty_r [0,1]^2$ for any integer $i\in [1,r+1]$;
 \item $\sum_{j=1}^{r+1} 2^{j-1} y_j =Y \mod 1$
 \end{enumerate}
\end{lem}

\begin{proof}
 Let $S=[0,1]^2$.
For fixed $x_1,\dots,x_{r+1}$ in the unit interval and any natural number $n$,
choose points $y_k(n)\in\x_r^nS(x_k)$ as provided by Lemma \ref{perenos}. There is an increasing sequence
of integers $n_i$ along which each sequence $y_k(n_i)$ converges to a limit $y_k$. These limits satisfy
the conditions of the lemma.
\end{proof}

For each natural number $r$, define the Stong map $\sigma_r\colon\mathbb{C}^{r+1}\to\R^{r+2}$ by
the formula
\begin{equation}\label{sigma}
 \sigma_r(z_0,z_1,\dots,z_{r})=(\re z_0,\dots,\re z_r,\sum_{k=0}^{r}2^{k}\im z_k)
 \end{equation}

Let $\tau_r$ be the map from the unit interval to an $r$-fold Cross inscribed in a unit square in
the complex plane, constructed in the proof of Theorem \ref{th-serp}.
Let $\tau_r^{r+1}\colon I^{r+1}\to\x_r^{r+1}\subset\mathbb{C}^{r+1}$ denote the $(r+1)$-fold Cartesian power of $\tau_r$.

\begin{theorem} The composition of $\tau_r^{r+1}$ and $\sigma_r$ maps the cube $[0,1]^{r+1}$ into
$[0,1]^{r+1}\times\R$. Composing this map with the covering $\pi\colon[0,1]^{r+1}\times\R\to[0,1]^{r+1}\times\R/\Z$
gives a surjective map with H\"older exponent $\frac{r+1}{r+2}$ and constant $(2^{r+1}-1)2^{(r+2)+(r+1)/(r+2)}$.
 \end{theorem}
\begin{proof}
Lemma \ref{lem-razvert} implies that the composition $\pi\circ\sigma_r\circ\tau_r^{r+1}$ is surjective.
By Theorem \ref{th-serp}, the map $\tau_r^{r+1}$ is H\"older-continuous with exponent $\frac{r+1}{r+2}$. Since the Stong map and the covering are Lipschitz, the composition has the same H\"older exponent.
The H\"older constant of the composition is the product of the H\"older constant of $\tau_r^{r+1}$,
namely $2^{(r+2)+(r+1)/(r+2)}$, and the Lipschitz constant of the Stong map, namely $(2^{r+1}-1)$,
because the covering map $\pi$ does not increase distances.
\end{proof}

\subsection{Cross and Squiggle}

\begin{lem} \label{arif2}
Let $(x_i^+,x_i^-)$ and $(y_i^+,y_i^-)$, $i=1,\dots,r$, be two sequences of pairs of integers such that
each pair $x_i^+,x_i^-$ has opposite parity, whereas each pair $y_i^+,y_i^-$ has the same parity but distinct residues modulo four.
Then every integer $N\in[0,4^r-1]$ has a unique representation
$N=\sum_{i=1}^{r}4^{i-1}(x_i^{\sigma(i)}+y_i^{\theta(i)})\mod 4^r$, where $\sigma(i),\theta(i)\in \pm$.
\end{lem}
\begin{proof}
Indeed, compare two sums of the stated form.
Reduction modulo four uniquely determines $x_1$ and $y_1$.
Next, reduction modulo $4^2$ uniquely determines $x_2$ and $y_2$, and so on.
The number of such sums equals the number of residue classes modulo $4^r$,
which proves the claim.
\end{proof}

\begin{lem}
\label{frac-seq2} Let $(b_1,\dots,b_r)$ be a Boolean vector, let $Y$ be an integer, and let $(C_1,C\prime_1),\dots,(C_r,C\prime_r)$ be pairs of columns of order $n\ge1$,
where each $C_i$ is a column of a Squiggle of multiplicity $2r-1$ and each $C_i\prime$ is a column of a Cross of multiplicity $2r-1$.

Then there exist cells $F_i\in C_i$ and $F_i\prime\in C_i\prime$ such that $\row_2F_i=b_i$ for every $i$, and
the following equality holds:
 \begin{equation}\label{sum-eq}
  \sum_{j=1}^{r} 4^{j-1} (\Row F_j+\Row F_j') =Y \mod 4^{nr}.
 \end{equation}
\end{lem}
 \begin{proof}
We prove the lemma by induction on the order of the columns.

First, let $n=1$.
Since each $C_i\prime$ contains cells at every vertical parity and
$C_i$ contains pairs of cells with vertical parity $b_i$ and distinct quaternary residues,
Lemma \ref{arif2} allows us to choose cells $F_i\in C_i$ and
$F_i\prime\in C_i\prime$ such that
\begin{equation}\label{seqq1}
 Y= \sum_{i=1}^{r} 4^{i-1}( \Row F_i+ \Row F_i') \mod 4^{r}
 \end{equation}
This establishes the base case.

Assume the claim holds for columns of order $n\ge1$, and consider columns $C_i$ of order $n+1$.

The same argument used for first-order columns lets us choose
cells $F_i,F_i\prime$ in the order-$(n+1)$ columns so that the equality
holds:
\begin{equation}\label{seq2}
 Y= \sum_{i=1}^{r} 4^{i-1}( \row F_i+ \row F_i') \mod 4^{r}
 \end{equation}

Set $Y_1=4^{-r}\left(Y-\sum_{i=1}^{r}4^{i-1}(\row F_i+\row F_i\prime)\right)$.
By the induction hypothesis,
there are cells $G_i\in\pr C_i$ and $G_i\prime\in\pr C_i\prime$ with $\odd G_i=\odd\pr F_i$ such that
 \begin{equation}\label{step-n1}
 \sum_{i=1}^{r} 4^{i-1} ( \Row G_i+\Row G_i') =Y_1 \mod 4^{nr}
 \end{equation}
Since $G_i$ has the same type as $\pr F_i$, it contains a subcell $F_i\prime\prime\subset G_i$ of order $n+1$ for which
$C(F_i\prime\prime)=C(F_i)$ and $\row F_i=\row F_i\prime\prime$. By Lemma \ref{zag-tip},
$\Row F_i\prime\prime=\Row F_i\bmod 4^r$, so the equality in \eqref{step-n1} also holds for the cells $F_i\prime\prime$.

We may therefore assume, without loss of generality, that $G_i=\pr F_i$ for each $i$.
Replacing $G_i$ by $\pr F_i$ in \eqref{step-n1} and multiplying by $4^r$ gives
$$
\sum_{i=1}^{r} 4^{i-1} ( 4^{r}(\Row \pr F_i+\Row \pr F_i'))
=
 Y-\sum_{i=1}^{r} 4^{i-1}(\row F_i+\row F_i') \mod 4^{(n+1)(r)}
$$
Moving the sum from the right-hand side to the left gives the required equality for columns of order $n+1$.
\end{proof}

The following theorem follows from Lemma \ref{frac-seq2} in the same way that Theorem \ref{th-razvert}
follows from Lemma \ref{frac-seq}.

\begin{theorem}\label{th-razvert2}
For any $x_0,\dots,x_{2r-1}$ in the unit interval and any $Y\in[0,1]$,
there exist $y_0,\dots,y_{2r-1}$ such that the following conditions hold:
\begin{enumerate}
\item for every $i\in[0,2r-1]$, the point $(x_i,y_i)$ lies on the Squiggle $\w_{2r-1}$ inscribed in $[0,1]\times[0,2]$ when $r$ is odd, and on the Cross $\x_{2r-1}$ when $r$ is even;
 \item $\sum_{j=1}^{r} 4^{j-1} (y_{2j-1}+y_{2j}) =Y \mod 1$
 \end{enumerate}
\end{theorem}
\begin{proof}
The proof is left to the reader.
\end{proof}

Let $\tau_{2r-1}$ denote the parametrization of the $(2r-1)$-fold Cross by the unit interval constructed in the proof of Theorem \ref{th-serp}.
Let $\zeta_{2r-1}^r$ and $\tau_{2r-1}^r$ denote the $r$-fold Cartesian powers of the unit-interval parametrizations of the $(2r-1)$-fold Squiggle and
Cross, respectively.
Define $\chi_r=\zeta_{2r-1}^r\times\tau_{2r-1}^r$ as the product of these maps.
Then $\chi_r$ is a H\"older-continuous map with exponent $\frac{2r}{2r+1}$ from the cube $[0,1]^{2r}$
into $\C^{2r}$.

Define $\theta_r\colon\C^{2r}\to\R^{r+1}$ by the formula
\begin{equation*}
\theta_r(z_1,\dots,z_{2r})=(\re z_1,\dots,\re z_{2r},(\sum_{k=1}^{r} 4^{k-1}\im (z_{2k-1}+z_{2k})) \mod 1)
\end{equation*}

\begin{theorem}\label{th-xz} The composition of $\chi_r$ with $\theta_r$ maps $[0,1]^{2r}$ into $[0,1]^{2r}\times\R$. Composing this map with the covering
$\pi\colon[0,1]^{2r}\times\R\to[0,1]^{2r}\times\R/\Z$ gives a surjective map with H\"older exponent $\frac{2r}{2r+1}$.
 \end{theorem}
\begin{proof}
Theorem \ref{th-razvert2} implies that the composition $\pi\circ\theta_r\circ\chi_r$ is surjective.
Theorem \ref{th-serp} implies that $\chi_r$ is H\"older-continuous with exponent $\frac{2r}{2r+1}$.
Since the Stong and covering maps are Lipschitz, the composition has the same H\"older exponent.

 \end{proof}

\section{Conclusion.}

In addition to proofs, the article gives recurrence relations for chain codes that make it possible to implement
the fractal curves discussed here.
The author thanks I. V. Pashenkov for programming the Zigzag and the Cross.
Computer calculations provided initial evidence that the maps from a square onto a three-dimensional cylinder are surjective.
A rigorous theoretical proof of surjectivity required considerable effort and ultimately led to the writing of this article.

\end{document}